\documentclass[11pt, a4paper]{article}
\usepackage[a4paper,margin=2cm]{geometry}

\usepackage[english]{babel}
\usepackage[T1]{fontenc}
\usepackage{amsmath, amssymb,amsthm}
\usepackage{algorithm}
\usepackage{algpseudocode}
\usepackage{tikz}
\usepackage{hyperref}
\usepackage{dirtytalk}

\usepackage[font=small,width=.85\textwidth]{caption}
\usepackage{subcaption}

\usepackage{setspace} 
\usepackage{etoolbox} 

\newtheorem{theorem}{Theorem}
\newtheorem{definition}{Definition}
\newtheorem{example}{Example}
\newtheorem{lemma}{Lemma}

\usepackage{xspace}
\newcommand{\fspeo}{FSP-EO\xspace}

\usepackage{todonotes}

\usepackage{natbib}
 \bibpunct[, ]{(}{)}{,}{a}{}{,}%

\newcommand{\I}{{\mathcal{I}}} 
\newcommand{\X}{{\mathcal{X}}} 
\newcommand{\F}{{\mathcal{F}}} 
\newcommand{\instfeaturespace}{\F_\I} 
\newcommand{\solfeaturespace}{\F_\X} 
\newcommand{\instance}{I}
\newcommand{\solution}{x}

\newcommand{\instfeaturemap}[1][(\instance^\instanceindex)]{\phi_\I#1} 
\newcommand{\solfeaturemap}[1][(\instance^\instanceindex,\solution^\instanceindex)]{\phi_\X#1} 
\newcommand{\instfeatures}{[p]} 

\newcommand{\fsel}{S} 
\newcommand{\selectionset}{\fsel}

\newcommand{\x}{{x}}

\newcommand{\soldist}[1][(\instanceindex,\instanceindextwo)]{d_\X#1} 
\newcommand{\neighbour}{\ensuremath{\mathcal{N}}}
\newcommand{\sEpsStrict}[1][(\instanceindex,\selectionset)]{\ensuremath{\neighbour^<_\epsilon#1}}
\newcommand{\sEpsBorder}[1][(\instanceindex,\selectionset)]{\ensuremath{\neighbour^=_\epsilon#1}}
\newcommand{\sEpsNeighbour}[1][(\instanceindex,\selectionset)]{\ensuremath{\neighbour^\le_\epsilon#1}}
\newcommand{\sEps}[1][(\instanceindex,\selectionset)]{\ensuremath{\neighbour_\epsilon#1}}
\newcommand{\skOpt}[1][(\instanceindex,\selectionset)]{\neighbour_k^\textnormal{opt}#1}
\newcommand{\skPess}[1][(\instanceindex,\selectionset)]{\neighbour_k^\textnormal{pess}#1}
\newcommand{\skGen}[1][(\instanceindex,\selectionset)]{\neighbour_k#1}
\newcommand{\Inew}{I^\text{new}}

\newcommand{\bc}[2]{\ensuremath{\begin{matrix}#1\\#2\end{matrix}}}
\newcommand{\sel}[1]{\begingroup\setlength{\fboxsep}{1pt}\colorbox{lightgray}{$#1$}\endgroup}
\newcommand{\selbc}[2]{\sel{\bc{#1}{#2}}}

\usepackage{authblk}

\DeclareMathOperator*{\argmax}{arg\,max}
\DeclareMathOperator*{\argmin}{arg\,min}

\title{Feature Selection for Data-Driven Explainable Optimization}
\begin{document}

\author[1]{Kevin-Martin Aigner}
\author[2]{Marc Goerigk}
\author[2]{Michael Hartisch}
\author[1]{Frauke Liers}
\author[1]{Arthur Miehlich\thanks{Corresponding author. Email: arthur.miehlich@fau.de}}
\author[1]{Florian R{\"o}sel}

\affil[1]{Department of Data Science and Department of Mathematics, Friedrich-Alexander-Universität Erlangen-Nürnberg, Germany}
\affil[2]{Business Decisions and Data Science, University of Passau, Germany}

\date{}

\maketitle
\abstract{%
Mathematical optimization, although often leading to NP-hard models,
is now capable of solving even large-scale instances 
within reasonable time. However, the primary focus is often placed
solely on optimality. This implies that while obtained solutions are
globally optimal, they are frequently not comprehensible to humans, in
particular when obtained by black-box routines. In contrast,
explainability is a standard requirement for results in
Artificial Intelligence, but it is rarely considered in
optimization yet. There are only a few studies that aim to find solutions
that are both of high quality and explainable. In recent work, 
explainability for optimization was defined in a data-driven manner:  A  solution is considered
explainable if it closely resembles solutions that have been used in the past under similar circumstances.
To this end, it is crucial to identify a preferably small subset of
features from a presumably large set that can be used to measure instance similarity.
In this work, we formally define the feature
selection problem for explainable optimization and prove that its decision version is
NP-complete. We introduce mathematical models for optimized feature
selection. As their global solution requires significant
computation time with modern mixed-integer linear solvers, we employ
local heuristics. Our computational study using data that reflect
real-world scenarios demonstrates that the
problem can be solved practically efficiently for instances of
reasonable size.


}%



\section{Introduction}

\subsection{Motivation}

Mathematical optimization plays a crucial role in solving complex
real-world problems by providing optimal solutions, even
for large-scale instances. The research focus has traditionally been
on achieving optimality, basically neglecting explainability of
obtained solutions. In applications where human beings are affected, it is however essential not only to determine
optimal outcomes but also to ensure that these solutions are
explainable and understandable to decision-makers.

Without explanation, decision-makers may not
trust or understand the reasoning behind the outcomes. If the process of finding a solution is perceived as a black box, it risks being ignored, even if the solution itself is globally optimal. Explanations bridge the gap between mathematical rigor and practical usability, increasing the likelihood of their real-world adoption.

In this paper, we utilize the data-driven framework presented in \cite{aigner2024framework}, where a solution is considered
explainable if it closely resembles favorable solutions that have been used in the past under similar circumstances. 
The approach works as follows:
A set of instances along with historical solutions is given. Each instance is described via a set of features, such as resource availability, cost information, or contextual factors, like weather conditions or seasonality. A metric to quantify pairwise distances between instances is provided. Similarly, each solution is characterized by a set of features, including key indicators, structural properties, or other relevant attributes, with a corresponding metric to measure the similarity between solutions. To express the explainability of a solution, the most similar historical instances are identified based on instance feature distances within a predefined threshold. The explainability score is then computed as the sum of the solution feature distances between the current solution and the solutions associated with these most similar historical instances.


In \cite{aigner2024framework}, the instance features are assumed to be given as input. We extend this work by investigating how to optimally select only few high-quality features from a set of instance features, which are then used
 to measure instance similarity. This will allow for more sparse explanations, thereby improving comprehensibility by relying on only a small number of features.
The corresponding instance feature selection problem can be solved as a preprocessing step, after which the existing explainable optimization framework can be utilized.
While feature selection is a standard task in Artificial Intelligence (AI), only few works have considered this problem in optimization-related tasks, often integrated within some optimization model. In our setting, feature selection is not primarily concerned with eliminating redundant features or minimizing information loss. Instead, we aim to select features such that instances that are similar in the selected feature space also have similar solutions---offering a different perspective on what constitutes a good selection of features.

\emph{Our contributions:}
\begin{itemize}
\item We give a formal definition of the feature selection problem in explainable mathematical optimization.

\item We clarify its complexity by proving that it is NP-complete in general.

\item We present mixed-integer linear mathematical optimization models for optimal feature selection.

\item Computational results for shortest path problems with realistic data show that typically a small set of features suffices to explain a solution.

\item We demonstrate that the feature selection problem can be solved satisfactorily by local search heuristics for large-scale instances.

\end{itemize}

\emph{Outline:} We summarize the state-of-the-art in the
relevant literature on feature selection in
Section~\ref{sec:literature} before discussing the formal problem setting in Section~\ref{sec:setting}.
To ensure being self-contained, this also includes the framework of \cite{aigner2024framework} for data-driven explainability. In Section~\ref{sec:theory}, we show that the feature selection problem we study is NP-complete. Building on this, Section~\ref{sec:methods} discusses solution methods, beginning with
mathematical optimization models. As their global solution is typically too time-consuming, we also explain our local $K$-opt heuristic. 
Computational results are presented in
Section~\ref{sec:results}. Finally, we conclude with a short summary and future research
questions in Section \ref{sec:conclusion}. 


\subsection{Related Literature}\label{sec:literature}

Explainability has long been neglected in mathematical optimization but has gained increased attention in the past five years. One reason for this neglect might be that theoretically substantiated optimality conditions already provide proof that a found solution is optimal, but they lack comprehensibility for non-experts. Furthermore, sensitivity analysis can be seen as a tool for obtaining insights into a found solution but requires a strong mathematical background and understanding of the optimization model.
Consequently, the mathematical optimization community has become more aware that AI techniques used for optimization should be more comprehensible~\citep{EJOR_Special_Issue,de2023explainable} and several approaches evolved that use mathematical optimization  to increase explainability and interpretability in AI \citep{Aghaei_Azizi_Vayanos_2019,carrizosa2024mathematical}.
Methods to enhance the explainability of solutions obtained through classical optimization have only recently emerged, often drawing inspiration from AI.

Initial efforts in this domain focused on specific applications, such as argumentation-based explanations for planning~\citep{collins2019towards, oren2020argument} and scheduling problems~\citep{vcyras2019argumentation, vcyras2021schedule}. More general approaches for dynamic programming~\citep{erwig2021explainable}, multi-stage stochastic optimization~\citep{tierney2022explaining}, and linear optimization~\citep{kurtz2024counterfactual} followed. Furthermore, data-driven approaches, which require historic or representative instances, have been proposed to enhance explainability \citep{aigner2024framework,forel2023explainable}. In the special case of multi-objective optimization problems, where researchers and practitioners already understood the necessity of being able to provide explanations to decision-makers, explainability plays a particularly critical role \citep{sukkerd2018toward, misitano2022towards, corrente2024explainable}. For recent developments in the field of explainability in optimization, we also refer to the website by \cite{expopt}.

In the field of metaheuristics, efforts to improve explainability have led to approaches such as the use of surrogate fitness functions to identify key variable combinations influencing solution quality to rank variable importance \citep{wallace2021towards}.
In a paper by \cite{fyvie2023explaining}, the authors analyze generations of solutions from population-based algorithms to extract explanatory features for metaheuristic behavior. By decomposing search trajectories into variable-specific subspaces, they rank variables based on their influence on the fitness gradient, providing insights into the algorithm’s search dynamics. Building on this approach, candidate solutions are used to identify problem subspaces that contribute to constructing meaningful explanations \citep{10.1007/978-3-031-77918-3_13}. Recent advances in hyper-heuristics leverage visualizations to explain heuristic selection dynamics, usage patterns, parameterization, and problem-instance clustering across various optimization problems \citep{yates2024explainable}. Interestingly, the explainability of heuristics connects the concepts of explainability and interpretability, the latter being concerned with making the solution process—rather than the resulting solution—comprehensible \citep{rudin2019stop,GOERIGK2023Interpretable}.

On the general topic of feature selection, a substantial body of literature exists. We refer the interested reader to the numerous surveys on the topic, e.g.  \cite{li2017feature,rong2019feature,dhal2022comprehensive}.
Here, we concentrate on contributions that are specifically relevant to our context, i.e., that either utilize mathematical optimization models to select relevant features, or focus on the use of features in a mathematical optimization setting. 

The idea of using mathematical models for tackling the problem of selecting relevant features is not new. Early approaches aimed at finding planes to separate two point sets with few non-zero entries \citep{bradley1998feature}. This was extended for more general support vector machines \citep{NGUYEN2010584,LABBE2019}. Furthermore, optimization models to select features for additive models \citep{navarro2025feature} and neural networks \citep{zhao2023model} have been developed. In \citep{bugata2019weighted}, the authors introduce a  feature selection method using $k$-nearest neighbors with attribute and distance weights optimized via gradient descent. This approach predicts relevant features, selecting the top features based on optimized weights that directly influence prediction accuracy. It is also noteworthy that already for the process of finding relevant features, approaches have been developed that aim at maintaining explainability \citep{zacharias2022designing} or that balance explainability and performance \citep{lazebnik2024algorithm}.

In the pursuit of making optimization more comprehensible, the role of meaningful features has recently gained importance, whereas their significance for both the explainability of a solution as well as the effectiveness of the solution process has long been acknowledged in AI \citep{janzing2020feature,johannesson2023explanatory}.  Notably, the term ``features'' is also sometimes referred to as contextual information, covariates, or explanatory variables. Feature-based approaches have been developed to enhance both the explainability \citep{aigner2024framework} and interpretability \citep{goerigk2024feature} of mathematical optimization.
For the data-driven newsvendor problem, a bilevel optimization approach for feature selection is proposed in \cite{SERRANO2024703}, where the leader validates a decision function and the follower learns it. The leader can restrict the decision function’s non-zero entries, influencing the feature selection. Unlike our approach, their focus is on improving out-of-sample performance and while their method results in explainable solutions, the primary goal is performance enhancement rather than explicit explainability.

\section{Problem Setting\label{sec:setting}}

\subsection{Framework for Explainability in Mathematical Optimization}



We first summarize the framework for data-driven explainability in mathematical optimization \citep{aigner2024framework}.
Consider a general mathematical optimization problem, without imposing assumptions on its specific structure or domain, so that the formulation can be applied to various settings.
 Let $\I$ denote the set of all instances for this problem, and let $\mathcal{X}$ represent the general domain of vectors that may be feasible. For each instance $I \in \I$, let $\X(I) \subseteq \mathcal{X}$ be the set of feasible solutions corresponding to $I$.
We denote by $f^I : \X(I) \rightarrow \mathbb{R}$ the objective function of instance $I$, and by $f^I(x)$ the objective value of a solution $x \in \X(I)$. The \emph{nominal} optimization problem for instance $I$, which seeks to minimize the objective function over its feasible set, is given by
\begin{equation}
    \min_{x \in \X(I)} f^I(x). \label{Eq:nom_problem} \tag{Nom}
\end{equation}
This formulation is very general in the sense that it can be applied to different types of optimization problems. It can accommodate both (discrete-)continuous as well as combinatorial problems such as the knapsack problem with any number of items or the shortest path problem with various graph structures.

The data-driven framework for explainability in mathematical optimization relies on the existence of high quality data consisting of historic instance-solution pairs.
Assume $N \in \mathbb{N}$ historic data points $(I^i, x^i)_{i \in [N]}$ are given, where $I^i \in \mathcal{I}$ is a full description of the $i$-th historic instance and $x^i \in \mathcal{X}(I^i)$ is the employed solution.  In the original framework, each data point carried an additional weight indicating whether the implemented solution was viewed positively or negatively for that instance. For this paper, we omit these weights and simply assume all historic data points have been 
assigned the same weight.

We consider feature functions $\instfeaturemap[]:\I\rightarrow \instfeaturespace$ and $\solfeaturemap[]:\I \times \X \rightarrow \solfeaturespace$ 
that aggregate information of instances as well as solutions within features spaces $\instfeaturespace \subseteq \mathbb{R}^p$ and $\solfeaturespace \subseteq \mathbb{R}^q$, respectively. We define two metrics $d_{\I}:\instfeaturespace \times \instfeaturespace \rightarrow \mathbb{R}_{\ge 0}$ and $\soldist[]:\solfeaturespace \times \solfeaturespace \rightarrow \mathbb{R}_{\ge 0}$ as (dis)similarity measures for instances and solutions, respectively.

Now, consider a new instance $\Inew \in \I$ for which we want to find an explainable solution.
Let
\begin{equation}
	\sEps[(\Inew)]=\{i \in [N] \mid d_\I(\instfeaturemap[](\Inew),\instfeaturemap[](I^i)) \leq \epsilon\}
	\label{s_eps}
	\tag{Sim-Set}
\end{equation} 
be the set containing the most similar historic instances with threshold $\epsilon\geq 0$, where $[N]:=\{1,\ldots,N\}$ denotes the index set.

The following bicriteria optimization model was proposed:
\begin{equation}
	\min_{x \in \X(\Inew)} \ \left( f^{\Inew}(x), \sum_{i \in \sEps[(\Inew)]}\frac{d_{\X}\left( \solfeaturemap[](\Inew,\x),\solfeaturemap[](I^i,\x^i)\right)}{1+ d_\I\left(\instfeaturemap[](\Inew),\instfeaturemap\right)}
	\right). \tag{Exp}\label{exp_formula}
\end{equation}

The first objective is the objective function of the nominal problem.
The second objective represents the explainability of solution $x$.
To calculate this value, consider all similar historic instances in $I^i \in \sEps[](\Inew)$ with solutions $x^i$.
We would like to obtain a solution that is similar to the historic solutions $x^i$ for each  instance $I^i$.
Similarity is measured in the solution feature space $\solfeaturespace$ using the distance $\soldist[(\solfeaturemap[(\Inew,\solution)],\solfeaturemap[(I^i,\solution^i)])]$.
This definition of explainability means that a decision-maker can point to historic examples, and explain the current choice based on similar situations in which similar solutions were chosen.

A key assumption of the approach is that the used
features indeed represent easily comprehensible aspects of both the
instance and the solution.
Due to the subjectivity of explainability
itself, quantifying this property is nontrivial in general. 
Furthermore, it is not clear what features are suitable to represent instance similarity. 
In this work, we determine optimal instance features to make good
solutions well explainable. To this end, we next present an example for the feature selection problem considered here. Subsequently, we formally introduce the feature selection problem in explainable optimization in Section~\ref{sec:theory}.

\subsection{Example for Feature Selection in Explainable Optimization}

To provide a clearer idea of the concepts to come, we present a detailed, prototypical example based on the classical knapsack problem, which serves to illustrate the feature selection task in our framework. Assume we should allocate some spending budget to a range of potential projects, which can be grouped into three sectors: infrastructure, education, and healthcare. Each project is associated with certain costs and benefits, reflecting, for example, its societal impact, contribution to public satisfaction, or financial return.
The task is to select a subset of projects that maximizes the overall benefit without exceeding the available budget. In order to be able to explain the spending plan, we utilize the framework summarized in the previous section. To this end, distance measures are required for both instances and solutions, based on their respective feature representations. We consider two solution features that are used to describe implemented solutions:
\begin{itemize}
	\item (numeric) solution feature $\phi_{\mathcal{X},1}: \mathcal{I}\times \mathcal{X} \rightarrow [0,1]$: project selection rate, i.e., the number of selected projects divided by the total number of proposed projects.
	\item (binary) solution feature $\phi_{\mathcal{X},2}: \mathcal{I}\times \mathcal{X} \rightarrow \{0,1\}$: indicator whether at least half of the projects from the sector with the highest cost-benefit ratio are selected in the solution, reflecting whether the solution includes a significant portion of the most valuable group.
\end{itemize}
For now, we assume that instance similarity is purely measured via the available budget.
Then, an explanation might look as follows: ``The last time we had a budget this large, we supported 40\% of all proposed projects, yet funded fewer than half of the projects from the sector that seemed most beneficial. Hence, we should do it like this again.''

Table~\ref{tab-example:instance_description} shows data from $N=4$ situations in the past where the implemented budget allocation was viewed favorably. It contains information on both the instance (costs $c$, benefits $b$, and budget given in arbitrary units) as well as the implemented solution (selected projects). 

\begin{table}[htbp]
	\caption{Description of the four historic instances $I^i \in \mathcal{I}$ with benefits $b$ and costs $c$, as well as the implemented solution $x^i$ (projects highlighted in gray were selected). The last row contains a new instance we want to find an explainable solution for. There, the marked projects maximize the overall benefit,  but this solution might be hard to explain.
		\label{tab-example:instance_description}} 
	\centering
	\begin{tabular}{c|c@{\hspace*{2em}}c@{\hspace*{2em}}c|c|cc}
		& \multicolumn{3}{c|}{projects} & budget&\multicolumn{2}{c}{solution features}\\
		& infrastructure & education & healthcare && $\phi_{\mathcal{X},1}$&$\phi_{\mathcal{X},2}$\\
		\hline
		$I^1$ \quad $\bc{b}{c}$ 
		& \bc{2}{1}\quad \bc{2}{1}\quad \bc{4}{2}
		& \selbc{9}{3}\quad \bc{9}{4} \quad \bc{10}{5}
		& \bc{11}{5}\quad \selbc{7}{2} 
		& 5 & 0.25 & 0\\[.55cm]
		
		$I^2$ \quad $\bc{b}{c}$ 
		& \bc{15}{10}\quad \selbc{22}{7}\quad \bc{21}{8}
		& \bc{20}{8}\quad \selbc{16}{5}
		& \bc{8}{5}\quad \bc{9}{4} \quad \bc{11}{4}
		& 14& 0.25 & 0\\[.55cm]
		$I^3$ \quad $\bc{b}{c}$ 
		& \selbc{2}{1}\quad \bc{3}{3}
		& \selbc{5}{2}\quad \selbc{4}{2}\quad \bc{6}{3}
		& \bc{3}{2}\quad \bc{4}{3} \quad \selbc{3}{1}
		& 6& 0.5 & 1\\[.55cm]
		
		$I^4$ \quad $\bc{b}{c}$ 
		& \selbc{9}{5}\quad \selbc{3}{1}\quad \bc{3}{4}
		& \bc{6}{4}
		& \selbc{8}{3}\quad \selbc{7}{3}\quad \bc{9}{7} 
		& 12& 0.57 & 1\\[.55cm]\hline
		$\Inew$ \quad $\bc{b}{c}$ 
		& \selbc{3}{2}\quad \selbc{6}{3}\quad \selbc{10}{5} \quad \bc{10}{6}
		& \bc{11}{6} \quad \bc{5}{3}
		& \bc{11}{8}\quad \selbc{13}{6} 
		& 16 &  0.50 & 1\\
	\end{tabular}   
\end{table}

Using the two solution features, we can compute pairwise distances between the historic solutions, e.g., using the 1-norm as shown in Table~\ref{tab-example:solution features distance}. We observe that instances~$I^1$ and~$I^2$, as well as~$I^3$ and~$I^4$, have highly similar solutions.\\

\begin{table}[htbp]
	\centering
	\caption{Solution distances of historic solutions using the 1-norm. 
		\label{tab-example:solution features distance}}
	\begin{tabular}{c||c|c|c|c}
		$i$,$j$ & 1 & 2 & 3 & 4 \\
		\hline\hline
		1 & & 0 & 1.25 & 1.32\\
		\hline
		2 & 0 &  & 1.25 & 1.32 \\
		\hline
		3 & 1.25 & 1.25 & & 0.07\\
		\hline
		4 & 1.32 & 1.32 & 0.07 &
	\end{tabular}
\end{table}

Now assume we encounter a new instance, shown at the bottom of Table~\ref{tab-example:instance_description}: It is the start of a new year and the spending plans must be specified. Calling a mathematical optimization solver on the knapsack instance yields the solution specified in the table with an overall benefit of $32$. Is this solution explainable? If we compare this solution to the solution implemented in the most similar instance from the past, we consult instance $I^2$, as here the difference between the budgets are smallest (we recall that we currently compare instances solely based on their budget value). However, both solutions have rather different characteristics as shown by their respective features, resulting in a solution distance of 1.25. Hence, with respect to the solution similarity (and hence the explainability), a better solution would be to select both healthcare projects and the low-cost infrastructure project, yielding an overall benefit of $27$, solution features $\phi_{\mathcal{X},1}=0.375$ and $\phi_{\mathcal{X},2}=0$, and a distance of $0.125$ to the solution implemented in instance $I^2$.


This discussion already shows that it is not trivial to decide which instances should be considered ``similar''. It is thus necessary to decide on appropriate instance features for which such an evaluation can be performed.
This work investigates which instance features are most appropriate for assessing instance similarity. As illustrated in the first part of this example, relying solely on the budget as an instance feature may not be an adequate approach. Let us assume that a list of reasonable instance features is available. For this example let there be five such features, for which the respective feature values of the five instances are shown in Table~\ref{tab-example:instance features}:

\begin{itemize}
	\item numeric instance feature $\phi_{\mathcal{I},1}$: available budget,
	\item categorial instance feature $\phi_{\mathcal{I},2}$: project group with the highest average cost–benefit ratio,
	\item binary instance feature $\phi_{\mathcal{I},3}$: indicates whether the overall average cost–benefit ratio exceeds~2,
	\item numeric instance feature $\phi_{\mathcal{I},4}$: overall number of proposed projects,
	\item numeric instance feature $\phi_{\mathcal{I},5}$: ratio between the average benefits of the best and the second-best group.
\end{itemize}

\begin{table}[htbp]
	\centering
	\caption{Instance feature values\label{tab-example:instance features}}
	\begin{tabular}{c|c|c|c|c|c}
		& $\phi_{\mathcal{I},1}$ & $\phi_{\mathcal{I},2}$ & $\phi_{\mathcal{I},3}$ & $\phi_{\mathcal{I},4}$ & $\phi_{\mathcal{I},5}$ \\
		\hline
		$I^1$ & 5 & healthcare & 1 & 8 & 1.04 \\
		$I^2$ & 14 & education & 1 & 8 & 1.07 \\
		$I^3$ & 6 & education & 0 & 8 & 1.50 \\
		$I^4$ & 12 & healthcare & 0 & 7 & 1.33 \\
		\hline
		$\Inew$ & 16 & infrastructure & 0 & 8 & 1.50
	\end{tabular}
	
\end{table}

We aim to identify a subset of instance features such that instances that are similar with respect to these features also yield similar solutions, revealing meaningful similarities and dissimilarities within the dataset.
This subset is derived solely from historical data and serves as a general similarity measure that can be applied to new, unseen instances, rather than being fine-tuned for each individual case. 
In this example, we consider selecting two out of the five available instance features. We note that the explanation becomes easier to follow---though not necessarily more consistent---when the number of features used in the explanation is kept small.

In this illustrative example, the most informative features are instance features $\phi_{\mathcal{I},3}$ and $\phi_{\mathcal{I},5}$. When similarity between instances is measured based on these two features, instance~$I^1$ is closest to~$I^2$ (and vice versa), while instance~$I^3$ is closest to~$I^4$. This is desirable, as these pairs also exhibit highly similar solutions. When these features are used to assess similarity between new and historical instances, the most similar historical instance to the new instance in our initial example is $I^3$, with a distance of $0$. It turns out that the optimal knapsack solution of the new instance (see Table~\ref{tab-example:instance_description}) has the same features as the solution implemented in instance $I^3$. Therefore, if instance features $\phi_{\mathcal{I},3}$ and $\phi_{\mathcal{I},5}$ are used to measure instance similarity, the optimal knapsack solution also corresponds to the best explainable solution. The next section formalizes these considerations.

\subsection{Feature Selection in Explainable Optimization}

Suppose we are given $p \in \mathbb{N}$ candidate instance features. Our goal is to select at most $L \in \mathbb{N}$ of them, so that instance similarity---measured using only the selected features---aligns well with similarity between the corresponding solutions.

To this end, consider $N$ historic data points $(I^i, x^i)_{i \in [N]}$. Let $\selectionset \subseteq [p]$ with $|\selectionset| \leq L$ denote the set of selected instance features. We write $d_{\mathcal{I}}^{\selectionset}$ for the projection of the function $d_{\mathcal{I}}$ onto the feature coordinates indexed by $\selectionset$. For ease of notation, we write $d_\I^{\selectionset}(i,j)$  for the instance distance between data points $i,j \in [N]$ under this projection.

This way, the selection of features determines the distances measured between instances and, consequently, the $k$ most similar instances associated with each data point. Since these most similar instances provide the reference solutions used to assess explainability, changing which instance features are selected directly influences the solution distances that enter the explainability objective.


\begin{definition}[Strict neighbors, borderline neighbors, neighbors]~\\
Let $\varepsilon > 0$ and a feature selection $\selectionset \subseteq [p]$ be given.
For any data point $(I^i, x^i)$, we define the set of \emph{strict neighbors} as
\[
\sEpsStrict
  = \{ j \in [N] \setminus \{i\} \mid d_{\mathcal{I}}^{\selectionset}(i,j) < \varepsilon \},
\]
the set of \emph{borderline neighbors} as
\[
\sEpsBorder
  = \{ j \in [N] \setminus \{i\} \mid d_{\mathcal{I}}^{\selectionset}(i,j) = \varepsilon \},
\]
and the set of \emph{neighbors} as
\[
\sEpsNeighbour
  = \mathcal{N}^{<}_{\varepsilon}(i,\selectionset) \cup \mathcal{N}^{=}_{\varepsilon}(i,\selectionset).
\]
\end{definition}


We now compare the solution $x^i$ of any data point $(I^i, x^i)$ with the solutions of the $k$ instances most similar to $I^i$ from the data. To this end, we further simplify notation and write $d_\X(i,j) = d_\X(\phi_\X(I^i,x^i),\phi_\X(I^j,x^j))$ to denote the solution distance of any two data points $i,j\in [N]$.
 It may happen, that the $k$ nearest neighbors are not uniquely determined because several instances have the same distance to $I^i$. To resolve this ambiguity, we use two possible tie-breaking rules: we either select those instances whose solutions are most similar to $x^i$ (the optimistic case), or those whose solutions differ the most (the pessimistic case).

\begin{definition}[Optimistic neighbors, pessimistic neighbors\label{def:neighbors}]
Let a selection of features $\selectionset\subseteq[p]$ be given and consider some data point $(I^i, x^i)$. Let $k \in \mathbb{N}$ be fixed and $\varepsilon > 0$ such that $|\sEpsStrict| < k \leq |\sEpsNeighbour|$. Let $\bar{k}=k-|\sEpsStrict|$ be the number of instances that must be selected from the borderline neighbors.
We define the set of \emph{optimistic neighbors} with respect to $\selectionset$ by
\[ \skOpt = \sEpsStrict \cup \argmin_{\substack{J \subseteq \sEpsBorder \\ |J| = \bar{k}}} \sum_{j\in J} d_\X(i,j). \]
We define the set of \emph{pessimistic neighbors} with respect to $\selectionset$ by
\[ \skPess = \sEpsStrict \cup \argmax_{\substack{J \subseteq \sEpsBorder \\ |J| = \bar{k}}}  \sum_{j\in J} d_\X(i,j). \]
If the minimizer or maximizer is not unique, an arbitrary tie-breaking rule is applied.
\end{definition}

We note that the pessimistic and optimistic problem versions coincide if all instance distance values $d_{\mathcal{I}}^{\selectionset}(i,j)$ are distinct (which gives a unique sorting), or if all solution distance values $d_\X(i,j)$ for $j \in \sEpsBorder$ are equal.

\begin{table}[htbp]
\centering
\caption{Notation of \fspeo.}\label{table:sets_params_vars}
\begin{tabular}{r|p{.55\textwidth}}
\textbf{Name} & \textbf{Description} \\
\hline
$\I$ & instance space \\
$\X$ & solution space\\
$N \in \mathbb{N}$ & number of historic data points  \\
\hline
$p \in \mathbb{N}$ & number of available instance features \\
$L \in \mathbb{N}$ & number of instance features to select \\
$ \selectionset \subseteq [p] $ & indicator set of selected instance features \\
$q \in \mathbb{N}$ &number of solution features \\
$\instfeaturespace \subseteq \mathbb{R}^p$ & instance feature space \\
$\F_\X \subseteq \mathbb{R}^q$ & solution feature space \\
$\phi_\I : \I \to \F_\I$ & instance feature map \\
$\phi_\X : \I \times \X \to \F_\X$ & solution feature map \\
$k \in \mathbb{N}$ & instance neighborhood size \\
$d_{\I}(i, j)$
& instance distance with respect to all features \\
$d_{\mathcal{I}}^{\selectionset}(i,j)$
& instance  distance with respect to selected features \\
$\soldist$
& solution  distance \\
\hline
$\sEpsStrict$/$\sEpsBorder$/$\sEpsNeighbour$   & set of strict/borderline/admissable neighbor instances of $I^i$\\
$\skGen^{\textnormal{opt}}, \skGen^{\textnormal{pess}}$ & set of $k$-nearest admissable optimistic/pessimistic neighbor instances of  $I^i$
\end{tabular}
\end{table}

The main symbols of our notation are summarized in Table~\ref{table:sets_params_vars}. 
We now formally define the feature selection problem as follows. 

\begin{definition}[Instance feature selection problem in explainable optimization]
	\label{def:FSP-EO}
Let $N$ historic data points $(I^i, x^i)_{i \in [N]}$, $k\in\mathbb{N}$, and $L\in\mathbb{N}$ be given.

The \emph{optimistic Instance Feature Selection Problem for Explainable Optimization} (\textbf{\fspeo}) is given by

\[
\min_{\substack{\selectionset\subseteq [p]\\ |\selectionset|\leq L}} \sum_{i\in[N]}
      \sum_{j\in \skOpt } d_{\mathcal{X}}(i,j).
\]

The \emph{pessimistic} \textbf{\fspeo} is given by

\[
\min_{\substack{\selectionset\subseteq [p]\\ |\selectionset|\leq L}} \sum_{i\in[N]}
      \sum_{j\in \skPess} d_{\mathcal{X}}(i,j).
\]
\end{definition}

\begin{example}
Given $N=3$ historic data points, we determine the optimal feature selection for the pessimistic and optimistic case. Figure~\ref{fig:exampleSk} illustrates the data points, consisting of shortest path instances and the corresponding optimal solutions.

\begin{figure}[h]
	    \begin{minipage}{0.33\textwidth}
		\centering
		\begin{tikzpicture}[node distance=1.7cm]
			\node[circle, draw, red] (S) {$s$};
			\node[circle, draw, red, right=of S] (T) {$t$};

			\draw[->, blue, line width=0.5mm, bend left=30] (S) to node[above]{$c_1=1$} (T);
			\draw[->, bend right=30] (S) to node[below]{$c_2=1.4$} (T);
			
			\node[above=0.5cm of S, anchor=south] (w1) {};
			\node[below=0.5cm of S, anchor=north] (w2) {};
			\path (S) -- node[midway, gray] {\large\(\instance^1\)} (T);
		\end{tikzpicture}
	\end{minipage}%
	\hfill
	\begin{minipage}{0.33\textwidth}
		\centering
		\begin{tikzpicture}[node distance=1.7cm]
			\node[circle, draw, red] (S) {$s$};
			\node[circle, draw, red, right=of S] (T) {$t$};

			\draw[->, bend left=30] (S) to node[above]{1.9} (T);
			\draw[->, blue, line width=0.5mm, bend right=30] (S) to node[below]{1.5} (T);
			
			\node[above=0.5cm of S, anchor=south] (w1) {};
			\node[below=0.5cm of S, anchor=north] (w2) {};
			\path (S) -- node[midway, gray] {\large\(\instance^2\)} (T);
		\end{tikzpicture}
	\end{minipage}%
	\hfill
	\begin{minipage}{0.33\textwidth}
		\centering
		\begin{tikzpicture}[node distance=1.7cm]
			\node[circle, draw, red] (S) {$s$};
			\node[circle, draw, red, right=of S] (T) {$t$};

			\draw[->, bend left=30] (S) to node[above]{3} (T);
			\draw[->, blue, line width=0.5mm, bend right=30] (S) to node[below]{1.4} (T);
			
			\node[above=0.5cm of S, anchor=south] (w1) {};
			\node[below=0.5cm of S, anchor=north] (w2) {};
			\path (S) -- node[midway, gray] {\large\(\instance^3\)} (T);
		\end{tikzpicture}
	\end{minipage}

\caption{Illustration of four historic data points in a shortest path setting. Edge labels indicate costs. The shortest $s$-$t$ path in each instance is highlighted in blue.}
\label{fig:exampleSk}
\end{figure}

The instance features are given by \(\phi_{\mathcal{I},1}=c_1\) and \(\phi_{\mathcal{I},2}=c_2\), and we define \(d_{\mathcal{I}}^{\selectionset}(i,j)\) as the absolute difference between the respective feature values. Consequently, we have \(p=2\) instance features and aim to select \(L=1\) of them. We further set \(k=1\), meaning that for each instance we only consider its nearest neighbor. The solution feature distance is set to 0 if two instances share the same solution and to 1 otherwise.

We now compute the objective value of \fspeo for both possible feature selections, each evaluated under the optimistic and the pessimistic case. We start by selecting the upper-edge feature. For each instance, we determine its nearest neighbor according to the selected feature and sum up the corresponding solution feature distances. For \(\instance^1\), the most similar instance is \(\instance^2\) with distance \(|1 - 1.9|\). Because their solutions differ, the resulting solution-feature distance is~1, both in the optimistic and in the pessimistic case.
For \(\instance^2\), the closest instance is \(\instance^1\) which leads to another increase of the objective value by 1 in both cases.
\(\instance^3\) has \(\instance^2\) as closest neighbor and as they share the same solution the objective value is not changed by that.

We now repeat the same procedure for the lower-edge feature. 
\(\instance^1\) and \(\instance^3\) share the same edge weight so they are closest to each other. As they differ in their solutions both results in a contribution of~1.
For \(\instance^2\) we have to distinguish between the optimistic and pessimistic case as both instances \(\instance^1\) and \(\instance^3\) are contained in $\sEpsBorder[]$. In the optimistic case we select \(\instance^3\), which minimizes the solution feature distance since both instances share the same solution. In the pessimistic case, we instead select \(\instance^1\),as it yields the largest solution feature distance among the candidates with equal instance feature distance. Consequently, the total distance contributed by \(\instance^2\) is 0 in the optimistic case but 1 in the pessimistic one.

Table~\ref{tab:exampleSk} summarizes the resulting computations for all combinations of feature selection and optimistic/pessimistic evaluation.

\begin{table}[h]
	\centering
	\begin{tabular}{c|c|c}
		& optimistic & pessimistic \\
		\hline
		upper edge & $1+1+0=\mathbf{2}$ & $1+1+0=\mathbf{2}$\\
		\hline
		lower edge & $1+0+1=\mathbf{2}$ & $1+1+1=\mathbf{3}$
	\end{tabular}
	\caption{Evaluation of the objective function of \fspeo. Each summand represents one instance and the solution distance to its closest neighbor with respect to the selected features.}
	\label{tab:exampleSk}
\end{table}


Consequently, when solving the optimistic \fspeo, the upper and the lower edge attain the same objective value, resulting in an undesired tie due to the deliberately small number of instances considered in this illustrative example. However, a closer inspection reveals that this identical value is induced by different instances being selected as closest neighbors. Thus, even though the optimistic evaluation yields the same outcome in this example, the underlying feature choice still differs. In contrast, the pessimistic \fspeo leads to a unique selection of the upper edge.

\end{example}

\section{Theory on Feature Selection for Explainable Optimization}
\label{sec:theory}



We show that the \fspeo is NP-hard by reduction from the Maximum
Coverage Problem, which is a generalization of the well-known
NP-complete Set Covering Problem. This implies that it is unlikely to find a polynomial-time solution algorithm for the \fspeo, and justifies to model the problem as mixed-integer program or to aim for solution heuristics. In the following proof, the pessimistic and optimistic problem versions coincide, so the hardness result holds for both versions.

\begin{theorem}\label{theorem:complexity}
The decision variant of the \fspeo is NP-hard.
\end{theorem}

\begin{proof}
We now reduce from the Maximum Coverage (MC) problem, which is known to be NP-complete \citep{feige1995threshold}.
As its input, a set $U$ of elements to be covered is given, along with a collection of subsets $T_1, T_2, \ldots, T_m \subseteq U$, and numbers $K$ and $W$. We need to decide whether there is a subset $C\subseteq[m]$ with cardinality at most $K$, such that $\cup_{j\in C} T_j$ contains at least $W$ elements. We assume without loss of generality, that all sets $T_j$ are pairwise different.

Given an MC instance, we generate an instance of \fspeo as follows.
For ease of notation, we split the set of historic data points into four different types of
elements (and identify them by their corresponding symbol). We use an instance $I_c$ that we call center instance
together with additional instances. The latter consist of $|U| + 1$
 instances
that we denote as $I_1, ..., I_{|U|+1}$, $|U|$
instances that we denote as $\bar{I}_1, ..., \bar{I}_{|U|}$, and
$|U|+1$  additional instances for each element in $U$ denoted as
$\bar{I}_1^0,\bar{I}_1^1,...,\bar{I}_{|U|}^{|U|-1},\bar{I}_{|U|}^{|U|}$. In
total, we obtain $N=1+|U|+1+|U|+|U|\cdot(|U|+1) = (|U|+2)(|U|+1)$
instances.

We create the following instance feature mappings for the different instances:
\begin{equation}
\instfeaturemap[(I)] := \begin{cases}
0_{m} &\text{ if }I = I_c,\\
\frac{1}{2K}1_m &\text{ if }I = I_i\text{, }i = 1,...,|U|+1,\\
\sum_{j: i \in T_j} e_j &\text{ if }I = \bar{I}_i\text{, }i=1,...,|U|,\\
\sum_{j: i \in T_j} e_j+\frac{1}{3K}1_m &\text{ if }I = \bar{I}^\ell_i\text{, }i=1,...,|U|,~\ell=1,...,|U|+1.
\end{cases}
\end{equation}
We denote by $0_m$ the $m$-dimensional vector containing $0$ in every
component, by $1_m$ the $m$-dimensional vector containing $1$ in every
component, and by $e_j$ the $j$-th unit vector in appropriate dimension.
We further create solution feature mappings for the different
instances as follows:
\begin{equation}
\solfeaturemap[(\solution^I)] := \begin{cases}
0_{|U|+2} &\text{ if }I = I_c,\\
e_i + e_{|U|+2} &\text{ if }I = I_i\text{, }i = 1,...,|U|+1,\\
e_i + 3e_{|U|+2} &\text{ if }I = \bar{I}_i\text{, }i=1,...,|U|,\\
2 e_{|U|+2} &\text{ if }I = \bar{I}^\ell_i\text{, }i=1,...,|U|,~\ell=1,...,|U|+1.
\end{cases}
\end{equation}
We now calculate the distances for $\soldist[]$ in Table~\ref{table:sol_dist}, where we always assume that $i\neq i^\prime$.

\begin{table}[H]
\centering
\begin{tabular}{c|ccccccc}
$\soldist[(\cdot,\cdot)]$ & $\x^{I_c}$ & $\x^{I_{i}}$ & $\x^{I_{i^\prime}}$ & $\x^{\bar{I}_{i}}$ & $\x^{\bar{I}_{i^\prime}}$ & $\x^{\bar{I}_{i}^{\ell^\prime}}$ & $\x^{\bar{I}_{i^\prime}^{\ell^\prime}}$\\
\hline\\[-2ex]
$\x^{I_c}$       &  0 & 2 & 2 & 4 & 4 & 2 & 2\\
$\x^{I_i}$       & 2 & 0 & 2 & 2 & 4 & 2 & 2 \\
$\x^{\bar{I}_i}$ & 4 & 2 & 4 & 0 & 2 & 2 & 2 \\
$\x^{\bar{I}_i^\ell}$ & 2 & 2 & 2 & 2 & 2 & 0 & 0
\end{tabular}
\vspace{.3cm}
\caption{Solution distances. }\label{table:sol_dist}
\end{table}

We note that we set $p=m$.
We complete the instance description by setting $L=K$ and $k=|U|$.
Let us assume that $C\subseteq[m]$ with $1\le|C|\le K$ is some feasible solution of the Maximum Coverage instance. Let $\gamma$ denote the number of elements that are not covered by this solution, i.e., $\gamma := |U \setminus \bigcup_{j \in C}T_j|$. We construct a set $\fsel\subseteq[p]$ by setting $\fsel = C$. By construction, $1\le|\fsel|\le L$. We now show that the objective value of $\fsel$ is equal to
\begin{equation}
2\gamma + 2k + 2(k+1)k + 2\gamma + 2k^2. \label{eq:npcobj}
\end{equation}

We can see this as follows. Each instance contributes to the objective value of the \fspeo by the solution distances of the $k$ closest instances according to the instance distance induced by $\fsel$. In the following, we denote this set of $k$ closest instances for instance $I$ as $\skGen[(I, \fsel)]$,

To calculate the total objective value, we proceed by considering each
instance (or group of instances) individually, finding its $k$ nearest neighbors, and summing the corresponding solution distances.
\begin{itemize}
	\item The center instance $I_c$ has distance $0 < \frac{|\fsel|}{2L}
          \leq \frac{1}{2}$ to instances $I_{i}$. It has distance
          $|\{j\in \fsel: i \in T_j\}|$ to instances $\bar{I}_{i}$.
          For the instances $\bar{I}_i^{\ell^\prime}$, the distance is
           $|\{j\in \fsel: i \in T_j\}| + \frac{|\fsel|}{3L}$.
          This means that the distance is
          $0$ to all instances $\bar{I}_i$ where
          element $i$ is not covered by $\bigcup_{j \in C}T_j$. Exactly $\gamma$ such instances exist. In contrast, the distance is at least $1$ to all instances $\bar{I}_i$ where element $i$ is covered by $\bigcup_{j \in
            C}T_j$. Hence, the $k$ closest neighbors to the center are
          $\gamma$ many instances $\bar{I}_i$ for elements $i$ that
          are not covered, and $k-\gamma$ many instances of other
          types. Each of these $k-\gamma$ neighbors has a solution distance of exactly $2$.
	Together, we obtain that $\sum_{I \in \skGen[(I_c, \fsel)]} \soldist[(\x^{I_c}, \x^{I})] = 4\gamma + 2(k-\gamma) = 2\gamma + 2k$.
	\item The instances $I_i$ always have the other instances $I_{i^\prime}$ as closest neighbors with an instance distance of $0$, independent of the choice of $\fsel$. All other distances are strictly larger than $0$. Hence, $I_i$ contributes a value of $\sum_{I \in \skGen[(I_i, \fsel)]} d_\mathcal{X}(\x^{I_i}, \x^{I}) = 2k$ to the solution distance $\soldist[]$ for all $i = 1,...,k+1$.
	\item Now consider instances $\bar{I}_i$. Their distance is $|\{j\in \fsel: i\in T_j\}|$ to the central instance $I_c$. Furthermore, their distance to $I_{i'}$ is
	\[ (1-\frac{1}{2L})|\{j\in \fsel: i\in T_j\}| + \frac{1}{2L}|\{j\in \fsel: i\notin T_j\} \ge \frac{|\fsel|}{2L} \geq \frac{1}{2L}. \]
	The distance of $\bar{I}_i$ to $\bar{I}_{i'}$ is
	\[ |\{j\in \fsel : i\in T_j, i'\notin T_j\}| + |\{j\in \fsel: i\notin T_j, i'\in T_j\}| \in \mathbb{Z}_{\geq 0}. \]
	There are at most $k-1$ instances $\bar{I}_{i^\prime}$ that have an instance distance of $0$ to $\bar{I}_i$.
	The distance of $\bar{I}_i$ to $\bar{I}^\ell_{i}$ is $\frac{|\fsel|}{3L} \in [\frac{1}{3L}, \frac{1}{3}]$, and finally, the distance to $\bar{I}^\ell_{i'}$ is
	\begin{align*}
		(1-\frac{1}{3L})|\{j\in \fsel: i\in T_j, i'\notin T_j\}|&\\
		+(1+\frac{1}{3L})|\{ j\in \fsel : i\notin T_j, i'\in T_j\}|&\\
		+\frac{1}{3L} (| \{ j \in \fsel: i \in T_j, i^\prime \in T_j \} | + | \{ j \in \fsel: i \notin T_j, i^\prime \notin T_j \} |) & \ge \frac{|\fsel|}{3L} \geq \frac{1}{3L}.
	\end{align*}
	We conclude the following:
	If $i$ is covered by $\cup_{j\in C} T_j$, then the $k$ most similar instances are either $\bar{I}_{i^\prime}$ or $\bar{I}^\ell_i$. Each of these has a solution distance of $2$, so the total distance with respect to $\soldist[]$ is $2k$. If $i$ is not covered, the most similar instances additionally contain the central instance $I_c$. In this case, the distance with respect to $\soldist[]$ is $2k+2$. Summing these values over all instances $\bar{I}_i$, we obtain a total solution distance of $2\gamma + 2k^2$.
	\item Finally, the instances $\bar{I}_i^\ell$ have only instances $\bar{I}_{i^\prime}^{\ell^\prime}$ as close neighbors. The sum of solution distances is $0$.
\end{itemize}


Adding all terms, we obtain Formula \eqref{eq:npcobj}. We conclude that if the MC problem has a solution which contains at least $k-\gamma$ elements, then the \fspeo instance we constructed has a solution with distance with respect to $d_X$ of at most $2\gamma + 2k + 2(k+1)k + 2\gamma + 2k^2$.
We finally see, that if there is a solution to \fspeo with $d_\mathcal{X}$ distance of at most $2\gamma + 2k + 2(k+1)k + 2\gamma + 2k^2$, then this corresponds to a solution of the MC problem with at least $k-\gamma$ elements. The selected features of the \fspeo correspond with the selected subsets of the MC problem.
\end{proof}

This proves NP-hardness of both variants of the \fspeo. Moreover, if the objective value of a selection  $\fsel \subseteq \instfeatures$ can be evaluated in polynomial time---which is the case if all distance functions are computable in polynomial time---the respective variant of \fspeo is NP-complete.

\section{Solution Approaches for the \fspeo}
\label{sec:methods}

We present mixed-integer programming (MIP) models for the optimistic and pessimistic variants of \fspeo in Section~\ref{sec:models}. We then introduce some heuristic algorithms in Section~\ref{sec:heuristic}. In the remainder of this paper, we focus on using the 1-norm for both instance and solution distance for better comprehensibility, but generalizations to other norms are straightforward.

\subsection{Mixed-Integer Programming Models for the \fspeo}
\label{sec:models}

The MIP models presented here for both the optimistic and the pessimistic \fspeo can be solved by available state-of-the-art software from mixed-integer programming. 

\subsubsection{The Optimistic \fspeo}\label{subsec:optimistic_fspeo_mip} 
We recall that the optimistic variant and the pessimistic variant differ in the tie-breaking rule employed in case that several borderline neighboring instances have the same distance: The optimistic variant of the problem is allowed to choose the instances with the lowest solution distance, whereas the pessimistic variant of the problem has to choose the instances with the highest solution distance.

In the following model, the binary variable $b_f$ indicates the selection of instance feature $f\in[p]$.
Continuous variables $d_{ij}$ denote the distance between instances $I^i$ and $I^j$, which depends on the selected features.

The binary variables $y_{ij}$ indicate, whether instance $I^j$ is part of the admissable neighbours of instance $I^i$.
The continuous variables $\epsilon_i$ are forced to be set such that $|\sEpsStrict| < k \leq |\sEpsNeighbour|$, i.e., it represent the threshold value for detecting the addmissable neighbours.
We write $d_f(i,j) = \big|(\phi_{\mathcal{I}}(I^i))_f - (\phi_{\mathcal{I}}(I^j))_f\big|$, which is
the absolute difference between the value of feature $f$ for instances  $I^i$ and $I^j$. The optimization model is given as follows. 
\begin{subequations}\label{problem:optimistic_fspeo}
	\begin{align}
		\label{problem:optimistic_fspeo:obj}\min \quad &\sum_{i=1}^N \sum_{j=1,j\neq i}^N\ y_{ij}\cdot \soldist\\
		\label{problem:optimistic_fspeo:k_neighbors}\text{s.t.} \quad &\sum_{j=1, j\neq i}^N y_{ij} \ge k\quad \forall i\in [N],\\
		\label{problem:optimistic_fspeo:L_features}&1 \le \sum_{f\in\instfeatures} b_f\le L,\\
		\label{problem:optimistic_fspeo:instance_diff}&d_{ij} = \sum_{f\in\instfeatures} b_f\ d_f(i,j) \quad \forall i\neq j\in [N],\\
		\label{problem:optimistic_fspeo:neighboring1}&d_{ij} \le\epsilon_i+(1-y_{ij})M\quad \forall i\neq j\in [N],\\
		\label{problem:optimistic_fspeo:neighboring2}&\epsilon_i\le d_{ij} + y_{ij}M\quad \forall i\neq j\in [N],\\
		\label{problem:optimistic_fspeo:vardef}&b\in\{0,1\}^{p},\; y\in\{0,1\}^{N\times N-1},\; d \in \mathbb{R}^{N\times N-1}_{\geq 0},\;
	    \epsilon\in\mathbb{R}_{\ge 0}^N.
	\end{align}
\end{subequations}

The objective function \eqref{problem:optimistic_fspeo:obj} is the sum over the solution distances of all neighboring instance pairs. 
Constraint~\eqref{problem:optimistic_fspeo:k_neighbors} makes sure that for each instance $I^i$ at least $k$ neighboring instances are selected.
Constraint~\eqref{problem:optimistic_fspeo:L_features} ensures that at most $L$ features are selected, i.e., that at most $L$ variables $b_f$ are set to $1$. A lower bound of $1$ is used to rule out the trivially optimal solution of selecting no features at all, resulting in an objective value of zero.
Constraint~\eqref{problem:optimistic_fspeo:instance_diff} models that the instance distance is equal to the sum of the selected feature distances. This formulation exploits the additivity of the 1-norm across coordinates, allowing the distance between two instances to be expressed as the sum of their featurewise differences. If a different norm is used, this constraint must be adapted accordingly.
Constraint~\eqref{problem:optimistic_fspeo:neighboring1} makes sure that the borderline neighboring instance distance of instance $I^i$, represented by variable $\epsilon_i$, is at least as high as the instance distance of instance $I^j$ to instance $I^i$ for all neighboring instances $I^j$.
Constraint~\eqref{problem:optimistic_fspeo:neighboring2} ensures that the borderline neighboring instance distance of instance $I^i$, represented by variable $\epsilon_i$, is at most as high as the instance distance of instance $I^j$ to instance $I^i$ for all non-neighboring instances $I^j$.
Constraints~\eqref{problem:optimistic_fspeo:vardef} define the domains of the model variables.

This optimistic problem variant may come with some shortcomings in certain settings. For example,
if there exists a meaningless feature (e.g., a feature in which all instances have the value 0), then it is optimal to select only this meaningless feature, as in that case, all instances have the same distance to each other and hence $\sEpsBorder = [N]\setminus\{i\}$, and the optimistic setting allows us to choose the $k$ instances with most similar solutions.
As an alternative, we now discuss the pessimistic problem variant.

\subsubsection{The Pessimistic \fspeo}
Our goal again is to select $L$ features that minimize the overall solution distances for the $k$ closest instances of each instance. Now, however, in case that the closest $k$ instances are not uniquely defined, the worst-performing ones are considered. 

We first study the evaluation problem for the case that the instance distances $d_{ij}$ are fixed to some positive values.
The evaluation problem then reads:
\begin{subequations}\label{problem:eval}
	\begin{align}
		\max \quad & \sum_{i=1}^N \sum_{j=1, j\neq i}^Ny_{ij} \cdot \soldist \\
		\label{problem:eval_bound}\text{s.t.}\quad & \sum_{j=1, j\neq i}^N d_{ij}y_{ij} \leq \min \Big\{ \sum_{j=1, j\neq i}^N d_{ij}y^\prime_{j} \colon \sum_{j=1, j\neq i}^N y^\prime_j \geq k, y^\prime \in [0,1]^{N-1} \Big\} \quad \forall i\in [N],\\
		\label{problem:eval_k}& \sum_{j=1,j\neq i}^N y_{ij} \leq k \quad \forall i\in [N],\\
		\label{problem:eval_integrality}& y_{ij}\in \{0, 1\}^{N \times N-1}.
	\end{align}
\end{subequations}
The binary variables $y_{ij}$ are equal to $1$ if instance $I^j$ is in $\skPess$.  
We want to note that Problem~\eqref{problem:eval} decomposes into $N$ smaller optimization problems, one for each $i \in [N]$. Each of these is a cardinality constrained knapsack problem.
\begin{lemma}
	The integrality constraints \eqref{problem:eval_integrality} are redundant as the continuous relaxation of \eqref{problem:eval} has always an integer solution.
\end{lemma}
\begin{proof}
We show that vectors with fractional $y_{ij}$ entries do not
  correspond to vertices of the linear programming relaxation.
	Let $\epsilon > 0$ be such that $ |\sEpsStrict| \leq k \leq |\sEpsNeighbour|$. Then the right hand side of Constraint~\eqref{problem:eval_bound} evaluates to
	\begin{equation*}
		\sum_{j \in \sEpsStrict} d_{ij} + (k - \lvert \sEpsStrict \rvert) \epsilon.
	\end{equation*}
	Hence, the variables $y_{ij}$ have value $1$ for all $j \in \sEpsStrict$, and have value $0$ for all $j\notin \sEpsNeighbour$.
	If further a variable $y_{ij}$ with $j \in \sEpsBorder$ has a fractional value, another index $j^\prime \in \sEpsBorder$ exists with $y_{ij^\prime}$ fractional due to Constraint~\eqref{problem:eval_k}. As $d_{ij} = d_{ij^\prime}$, increasing $y_{ij}$ by $\min\{(1 - y_{ij}), y_{ij^\prime}\}$ and decreasing $y_{ij^\prime}$ by the same number preserves feasibility. Hence, no fractional $y$ can be a vertex of the feasible set of the continuous relaxation of Problem~\eqref{problem:eval} and the integrality constraints can be removed without consequences.
\end{proof}

We dualize the inner minimization problem to replace Constraint~\eqref{problem:eval_bound} by
\begin{subequations}\label{problem:eval_dualize_inner}
	\begin{align}
		& \sum_{j=1, j\neq i}^N d_{ij}y_{ij} \leq k \pi_i - \sum_{j=1, j\neq i}^N \rho_{ij} \quad \forall i\in [N],\\
		&\pi_i - \rho_{ij} \leq d_{ij} \quad \forall i\neq j\in [N],\\
		& \pi_i \geq 0, \rho_{ij} \geq 0 \quad \forall i\neq j\in [N].
	\end{align}
\end{subequations}
and dualize Problem~\eqref{problem:eval}, considering Constraint~\eqref{problem:eval_k} as inequality, and express the $d_{ij}$ as a combination of at most $L$ instance features to get
\begin{subequations}\label{problem:eval_dualize_outer}
	\begin{align}
		\min \quad & \sum_{i=1}^N \sum_{j=1,j\neq i}^N d_{ij} \beta_{ij} + \sum_{i=1}^N k \gamma_i + \sum_{i=1}^N \sum_{j=1, j\neq i}^N \delta_{ij}\\
		\nonumber\text{s.t.} \quad & \eqref{problem:optimistic_fspeo:L_features}, \eqref{problem:optimistic_fspeo:instance_diff}\\
		\label{problem:eval_dualize_outer:b_cons}& d_{ij} \alpha_i + \gamma_i + \delta_{ij} \geq \soldist\quad \forall i\neq j\in [N],\\
		\label{problem:eval_dualize_outer:alpha_bound}& \sum_{j=1,j\neq i}^N \beta_{ij} \geq k \alpha_i \quad \forall i\in [N],\\
		\label{problem:eval_dualize_outer:beta_bound}& \alpha_i \geq \beta_{ij} \quad \forall i\neq j\in [N],\\
		&\alpha, \beta, \gamma, \delta, d \geq 0, \;b \in \{0, 1\}^{p}.
	\end{align}
\end{subequations}

We note that Problem~\eqref{problem:eval_dualize_outer} has quadratic terms in the objective function and Constraint~\eqref{problem:eval_dualize_outer:b_cons}. In order to reformulate the optimization problem as a mixed-integer linear program, we note that $d_{ij}$ can be replaced by binary variables using \eqref{problem:optimistic_fspeo:instance_diff}.
The product of binary variables $b_f$ and continuous variables $\alpha_i$ and $\beta_{ij}$ can be linearized using standard envelope inequalities, if
the continuous variables are bounded for all $i\neq j \in [N]$. As Constraint~\eqref{problem:eval_dualize_outer:beta_bound} bounds $\beta_{ij}$ by $\alpha_i$, it suffices to show that the $\alpha_i$ are bounded.
\begin{lemma}\label{lemma:bound_on_alpha}
	There is an optimal solution $(\alpha^*, \beta^*, \gamma^*, \delta^*, d^*, b^*)$ to Problem~\eqref{problem:eval_dualize_outer} that fulfills
	\begin{equation}
		\label{alphabound}\alpha_i^* \leq \max_{j \in [N]\setminus\{i\}} \frac{\soldist}{\min_{f\in \instfeatures: d_f(i,j) \neq 0} d_f(i,j)}
	\end{equation}
\end{lemma}
\begin{proof}
	Assume that Inequality~\eqref{alphabound} does not hold for  $i \in [N]$. Then we can construct another optimal solution by
        setting $\tilde{\alpha}_i^*$ to the right-hand side of the inequality, and $\tilde{\beta}^*_{ij}$ to $\frac{\beta^*_{ij}\tilde{\alpha}_i^*}{\alpha_i^*}$ for all $j\in [N]$. For this, we consider two sets $M^0 := \{j \in [N]\setminus \{i\} \colon d^*_{ij} = 0\}$ and $M^+:= \{j \in [N]\setminus \{i\} \colon d^*_{ij} \neq 0\}$. For $j \in M^0$, it holds that
	\begin{equation*}
		\gamma_i^* + \delta^*_{ij} \geq d_\X(\solution^i,\solution^j),
	\end{equation*}
	hence this constraint does not restrict our choice on $\alpha_i$. For $j \in M^+$ it holds that
	\begin{align*}
		d_{ij}^* \tilde{\alpha}_i + \gamma_i^* + \delta_{ij}^*& \geq \min_{f\in \instfeatures: d_f(i,j) \neq 0}d_f(i,j) \frac{\soldist}{\min_{l\in \instfeatures: d_f(i,j) \neq 0} d_f(i,j)} + \gamma_i^* + \delta_{ij}^*\\ 
		&\geq \min_{f\in \instfeatures: d_l(i,j) \neq 0}d_f(i,j) \frac{\soldist}{\min_{f\in \instfeatures: d_f(i,j) \neq 0} d_f(i,j)}\\ 
		&= \soldist. 
	\end{align*}
Since
	\begin{equation*}
		\sum_{j =1,j\neq i}^N \tilde{\beta}_{ij}^* = \sum_{j =1,j\neq i}^N \frac{\beta_{ij}^* \tilde{\alpha}^*_i}{\alpha_i^*} = \frac{\tilde{\alpha}^*_i}{\alpha_i^*} \sum_{j =1,j\neq i}^N \beta_{ij}^* \geq \frac{\tilde{\alpha}^*_i}{\alpha_i^*} k \alpha_i^* = k \tilde{\alpha}_i^*
	\end{equation*}
the new solution fulfills Constraint~\eqref{problem:eval_dualize_outer:alpha_bound}.
We calculate that with
	\begin{equation*}
		\tilde{\alpha}_i^* = \frac{\tilde{\alpha}_i^* \alpha_i^*}{\alpha_i^*} \leq \frac{\tilde{\alpha}_i^* \beta_{ij}^*}{\alpha_i^*} = \tilde{\beta}_{ij}^*,
	\end{equation*}
Constraint~\eqref{problem:eval_dualize_outer:beta_bound} is also satisfied. Hence the solution we constructed is feasible and fulfills Inequality~\eqref{alphabound}. As $\tilde{\beta}_{ij}^* \leq \beta_{ij}^*$ for all $j \in [N]\setminus \{i\}$, its objective value is not worse than the value of the original solution. This proves the claim. 
\end{proof}

Lemma~\ref{lemma:bound_on_alpha} ascertains that we can reformulate the terms $d_{ij}\beta_{ij}$, as well as $d_{ij}\alpha_i$, with standard techniques.
As it has turned out in preliminary computational results that Problems~\eqref{problem:optimistic_fspeo} and \eqref{problem:eval_dualize_outer} are notoriously hard to solve, we present heuristic approaches for both versions of the \fspeo in the next section.

\subsection{Heuristic Approaches}\label{sec:heuristic}


High-quality heuristics are both relevant in their own right and can additionally be helpful to determine good starting solutions for the MIP formulations of the \fspeo presented in the previous section.
Fortunately, the constraints on the decisions that have to be made for the \fspeo are relatively simple.
A feature selection is uniquely defined by a subset $\fsel\subseteq \instfeatures$, and its feasibility depends solely on the size.
Hence, the structure of the set of feasible solutions simplifies the application of heuristics such as genetic algorithms, simulated annealing or swapping heuristics.
Preliminary computational tests have shown that several variants of
the $K$-opt heuristic performs surprisingly well on the \fspeo.

To evaluate a single feature selection $\fsel$, we evaluate the objective function (here denoted as $D(s)$) of the \fspeo from Definition~\ref{def:FSP-EO}. This ensures that the heuristic is assessed using the same evaluation measure as the mixed-integer programs~\eqref{problem:optimistic_fspeo} and \eqref{problem:eval}, while the heuristic itself remains unchanged between the optimistic and pessimistic cases and only the evaluated objective function differs.
\begin{algorithm}
	\caption{$K$-opt heuristic for the \fspeo}
	\begin{algorithmic}[1]
		\Require Initial set of selected features $\fsel$, integer $K$
		\Ensure Improved selection $\fsel$
		\State $imp \gets$ true
		\While{$imp$}
		\State $imp \gets$ false
		\ForAll{$s \in \{s' \subseteq \instfeatures \colon |s' \triangle \fsel| \leq 2K, |s'| = |\fsel|\}$}	\label{alg:iter}
		\If{$D(s) < D(\fsel)$}
		\State $\fsel \gets s$
		\State $imp \gets$ true
		\EndIf
		\EndFor
		\EndWhile
		\State \Return Improved selection $\fsel$
	\end{algorithmic}
	\label{alg:k-opt}
\end{algorithm}

In Algorithm~\ref{alg:k-opt}, we always replace exactly $k$ selected features by $k$ unselected ones, such that the total number of active features remains constant. If more flexibility in the number of selected features is desired, the algorithm can be adapted to allow swaps of individual entries in the binary feature vector, changing values from 0 to 1 and vice versa, while ensuring that the total number of selected features does not exceed the prescribed maximum number of $L$ features. This extension enables a more flexible feature cardinality, but at the same time increases the risk of becoming trapped in local minima.

Returning to Algorithm~\ref{alg:k-opt}, it is already apparent that the number of possible feature combinations grows exponentially with the size of the instance-feature space. Without further precautions, this would lead to prohibitively long runtimes. To obtain good solutions within a reasonable time, we introduce several practical modifications to the basic algorithm. In the following, we describe the corresponding hyperparameter choices, which were fixed for all experiments and tuned based on preliminary runs.

Instead of testing all possible permutations (step \ref{alg:iter}), we sample a maximum of 1000 permutations. Particularly in the early iterations of the $K$-opt algorithm, many improving permutations can be found. Therefore, we implemented an additional termination criterion that immediately proceeds to the next iteration as soon as 10 improving permutations are found. This allows the algorithm to capitalize on the high density of improving moves early on, significantly reducing the overall runtime.

To obtain a reasonably good starting solution of selected features, we evaluate 10 random vectors each time and use the best one as the required initial feature selection.
Moreover, we run the entire algorithm from 5 different starting vectors and keep the best final result. 

All numerical parameter choices reflect settings that worked well in our experimental setup, but they are not intrinsic to the method and can be adjusted depending on instance size, available computation time, and desired solution quality.

\section{Computational Experiments}\label{sec:results}

Since the performance of the feature selection approach cannot be meaningfully assessed in isolation, we evaluate it within the explainable optimization framework for which it was originally developed.
In this section, we present the experimental setup used to evaluate the effectiveness of our feature selection approaches. The experiments are designed to assess both computational performance and explainability, using real-world traffic data from Chicago (see \cite{chassein2019algorithms,zenodo_chicago_data}). We analyze the performance on different graph structures, varying instance sizes, and feature sets to understand their impact on the explainability and quality of the obtained solutions. To ensure the robustness of our findings, we conduct experiments on multiple subsets of the data and compare our results against established benchmarks.

\subsection{Data, Features and Performance Evaluation}
\begin{figure}[htbp]
	\centering
	\begin{subfigure}[b]{0.45\textwidth}
		\centering
		\includegraphics[width=.7\textwidth]{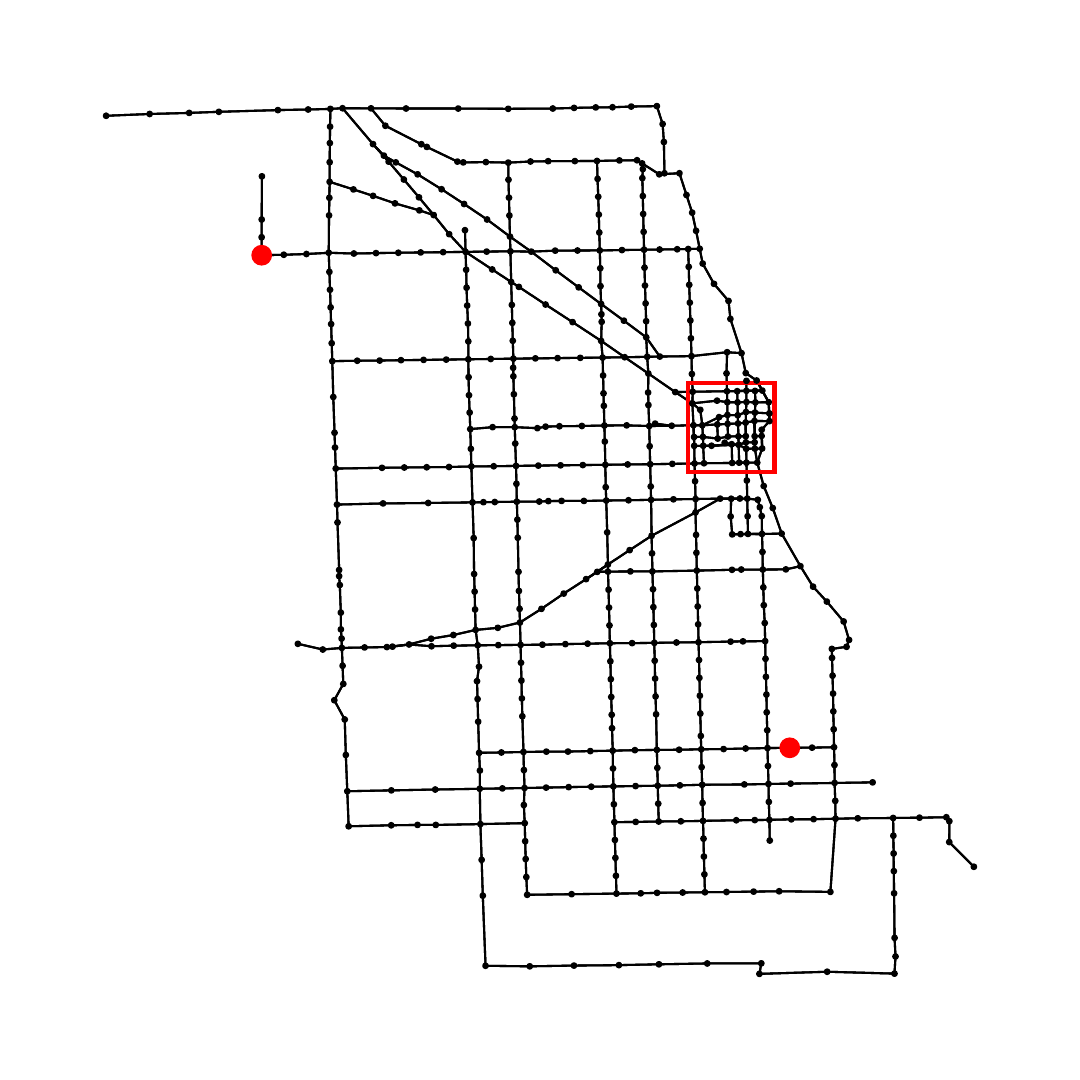}
		\caption{full graph	\label{fig:graph_full}}
	\end{subfigure}
	\hfill
	\begin{subfigure}[b]{0.45\textwidth}
		\centering
		\includegraphics[width=.7\textwidth]{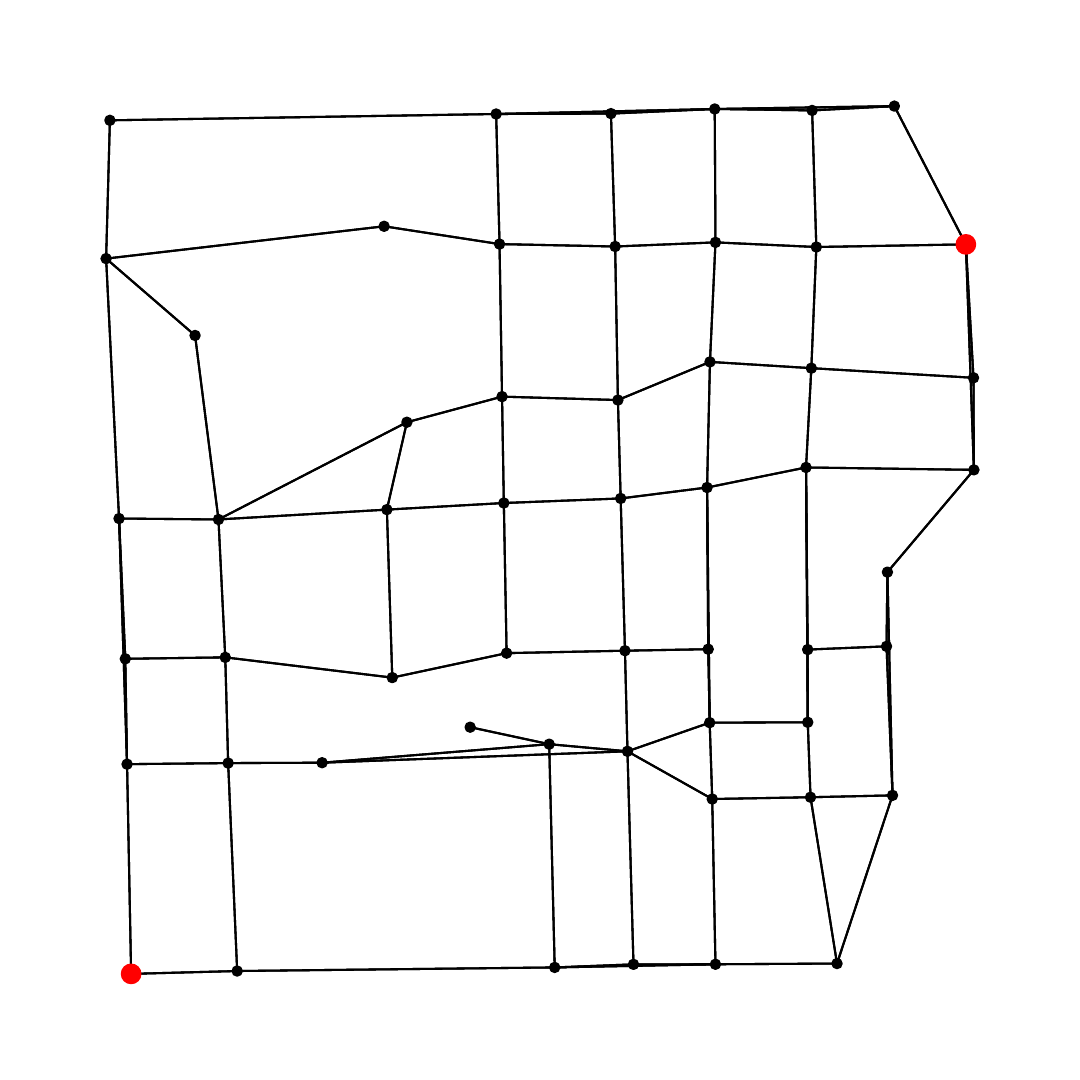}
		\caption{city center\label{fig:graph_cc}}
		
	\end{subfigure}
	\caption{Used graphs for the experiments. The start and end point of the shortest path are marked in red. The extracted city center is marked in the left\label{fig:used_graphs}}
	
\end{figure}
For our experiments, we consider the shortest path problem on a real-world road network derived from data collected from bus drivers in Chicago. The network topology consists of 538 nodes and 1,287 edges, where each edge is defined by the coordinates of its start and end points. Together, these nodes and edges represent the city’s road network, as illustrated in Figure~\ref{fig:graph_full}. For some experiments we focus on a subgraph  (the city center) with only 54 nodes and 196 edges (see Figure \ref{fig:graph_cc}). For both types of experiments we fixed the start and end point as shown in the respective figures. The dataset comprises 4,363 historical scenarios of average edge velocities, which we use as edge weights. Throughout our experiments, we use version~2 of the dataset as provided via Zenodo~\citep{zenodo_chicago_data}.
To complete the data required by the framework, for each of the 4,363 traffic scenarios we computed the corresponding optimal solution and represent it as a binary feature vector. Each entry of this vector corresponds to an edge and takes the value 1 if the edge is part of the optimal path, and 0 otherwise. 
To perform feature selection, we also required a list of candidate features. To this end, we accumulated data into  so called grid features, by dividing the graph into a grid of rows and columns as shown in Figure~\ref{fig:grid_features}. For each resulting cell, the weights of all edges whose midpoints fall within the cell were summed up. This ensures that every edge contributes to exactly one feature. The grid dimensions were chosen individually for each  experiment and will be specified later. Additionally, unless specified otherwise, the individual edges were also included as potential features. Features that exhibited the same value in all data points were removed from the list.
\begin{figure}[htbp]
	\centering
	\begin{subfigure}[b]{0.45\textwidth}
		\centering
		\includegraphics[width=.7\textwidth,height=.67\textwidth]{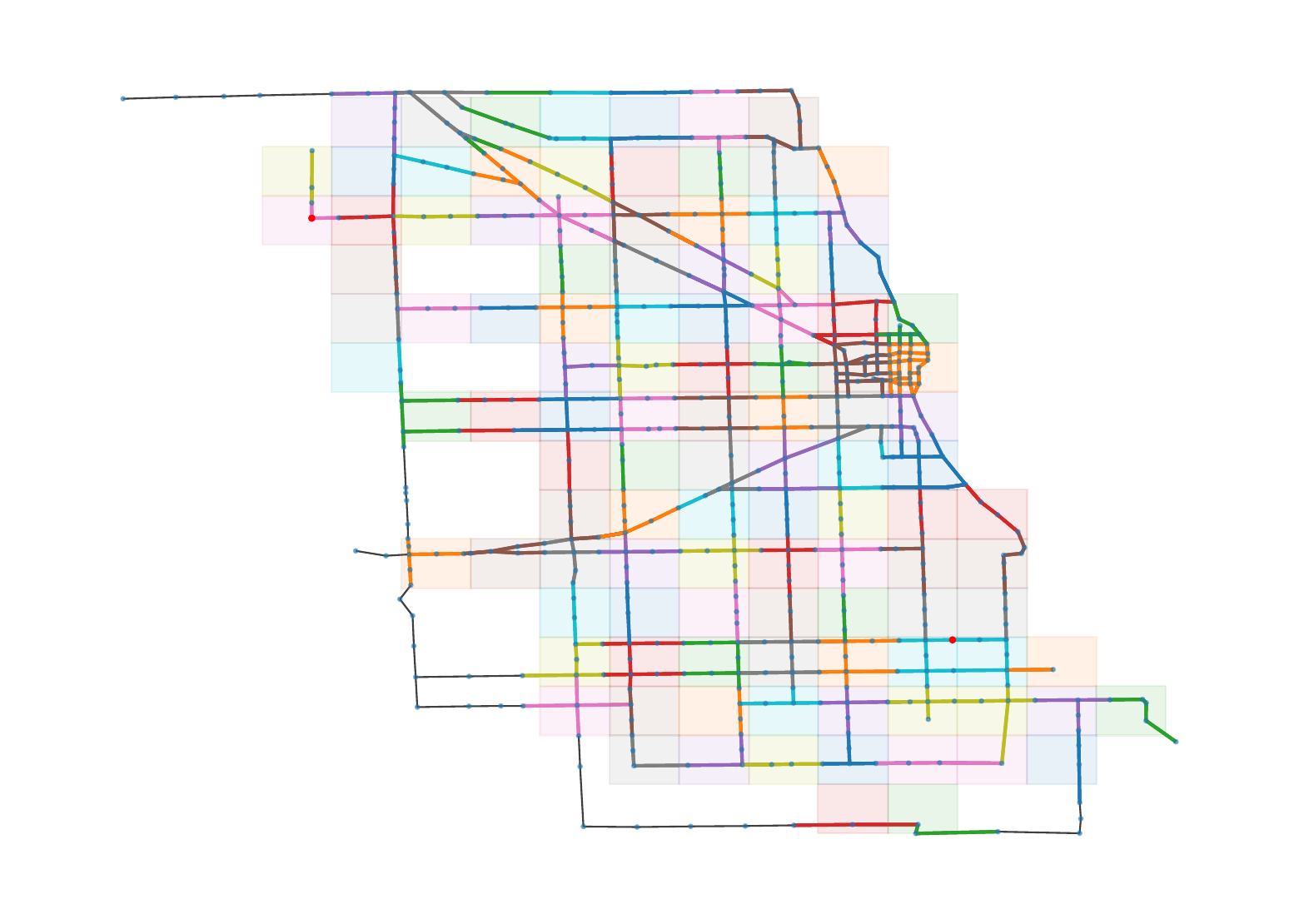}
		\caption{ 110 grid features on full graph}
		\label{fig:features_full}
	\end{subfigure}
	\hfill
	\begin{subfigure}[b]{0.45\textwidth}
		\centering
		\includegraphics[width=.7\textwidth,height=.67\textwidth]{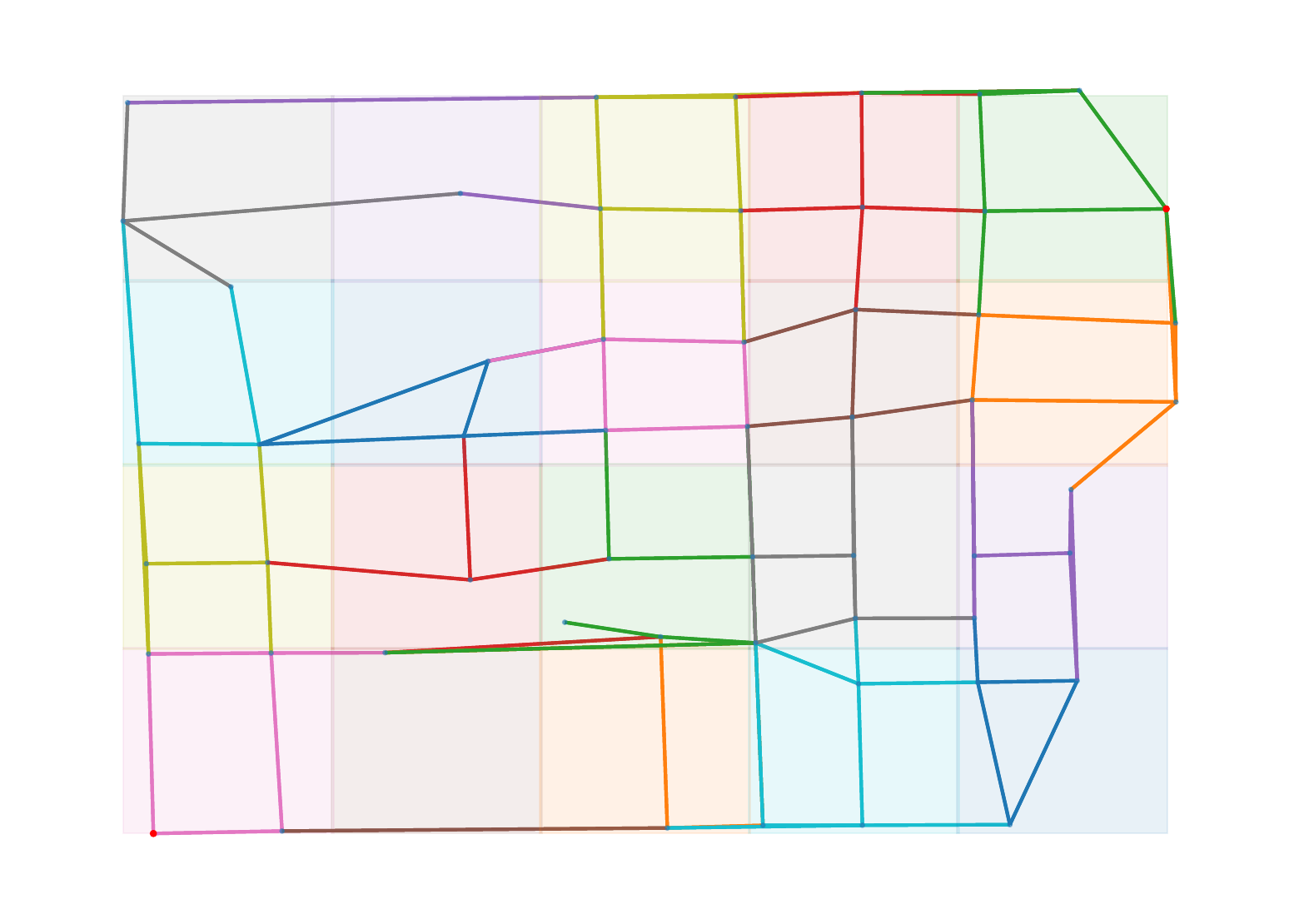}
		\caption{20 grid features on city center graph.}
		\label{fig:features_cc}
	\end{subfigure}
	\caption{Used grid features for the full graph and the subgraph. On the right, each edge is colored according to the grid cell that contains its midpoint.}
	\label{fig:grid_features}
\end{figure}



In the previous work \citep{aigner2024framework}, the bicriteria problem~\eqref{exp_formula} was addressed through a weighted-sum scalarization to generate a Pareto front illustrating the trade-off between optimality and similarity. In the present paper, however, our focus lies on the explainability aspect, for which similarity plays the central role. Consequently, we restrict our analysis to the variant of the problem that places full emphasis on similarity and disregards the optimality objective. We thus concentrate on the solution that aligns most closely with the most similar instances in $\sEpsNeighbour[]$ and compare its relative length to the optimal path to quantify the cost of explainability.

We then compare the performance of the solution obtained using the features selected by our model against solutions obtained by (i) using all individual edge weights as the instance feature vector, as done in \cite{aigner2024framework}, and (ii) randomly selecting $L$ instance features. The latter is repeated 100 times for each experiment, and average results are reported.




In all experiments, we set the neighborhood cardinality to  $k=5$. To ensure consistency in the experiments and to account for statistical fluctuations, we repeat each experiment ten times using different subsets of $N$ data point drawn from the dataset. The results are evaluated separately for each subset and then averaged.

All computations were performed on a high-performance cluster using a Python implementation.
The experiments were executed on a machine equipped with two Intel Xeon Gold 6326 (``Ice Lake'') processors, each featuring 2 $\times$ 16 cores clocked at 2.9 GHz. For solving the MIP models, we employed Gurobi 12.0.0.

\subsection{Results}
In this section, we evaluate our approaches on both small and large instances. All data is based on a real-world setting using realistic data.  

\subsubsection{First Experiment: Small-Scale Problems}

Here, we study a small-scale problem that allow us to obtain  optimal solutions via the proposed MIP formulations.
To this end, for each experiment we sample only $N=10$ data points from the data set and only use the city center graph with a reduced number of $2 \times 3$ grid features. In this simplified setup, the optimistic \fspeo (containing $2N^2$ big-M constraints) still can be solved to global optimality.
Interestingly, it turned out that  in this setting, also the $1$-opt
heuristic (see Section \ref{sec:heuristic}), consistently selects features that coincide with the optimal \fspeo results. However, this observation is likely due to the simplicity of the example and does not necessarily generalize to more complex instances.



When increasing the number of considered data points to $N=20$, the optimization model frequently exhibited a gap of 80--90\% after one hour of computation. For smaller problem sizes, this is likely due to slow improvements in the dual bound, whereas for larger instances, the 1-opt heuristic consistently outperforms the best \fspeo solutions obtained within the time limit. This indicates that best found solutions remain far from optimality in these cases. Similar challenges arise for the pessimistic \fspeo \eqref{problem:eval_dualize_outer}  which struggles to produce feasible solutions or exhibits large optimality gaps.
In the following, we therefore restrict our attention to the heuristic approach and evaluate it exclusively under the pessimistic tie-breaking rule. This choice reflects a worst-case perspective, as the pessimistic case penalizes equal nearest-neighbor assignments most strongly. In practice, however, ties between instance distances occur only very rarely in the considered real-world data sets, such that the difference between optimistic and pessimistic evaluations is negligible in most cases. Consequently, adopting the pessimistic variant provides a conservative evaluation of the heuristic without materially affecting the empirical results.

\subsubsection{Second Experiment: City Center Graph}
For the second experiment, we again used the smaller city center graph but increased the number of considered data points to $N=200$. The feature list, from which $L$ features were to be selected, included $4 \times 5$ grid features and additional edge features, resulting in a total of 121 candidate features.

In Figure \ref{fig:exp2} the $x$-axis represents the number of features that need to be selected and the $y$-axis shows the relative length of the most explainable path. We observe that even with a relatively
small number of features, our approach of first selecting relevant features often performs better than the original approach where costs of every single edge were used as features to find similar data points. Beyond producing better solutions, our approach increases explainability by limiting the comparison to few features (rather than all individual edge costs).
The random feature selection as the other benchmark leads to the worst solutions.

Interestingly, the selected features frequently included a balanced mix of individual edges and grid features. This suggests that a well-designed combination of both feature types could provide the best explainability. While the grid features were initially placed at equidistant intervals, a more detailed analysis of the city's structure could further refine their selection, potentially enhancing comprehensibility.

\begin{figure}[H]
	\centering
	\includegraphics[width=0.5\linewidth]{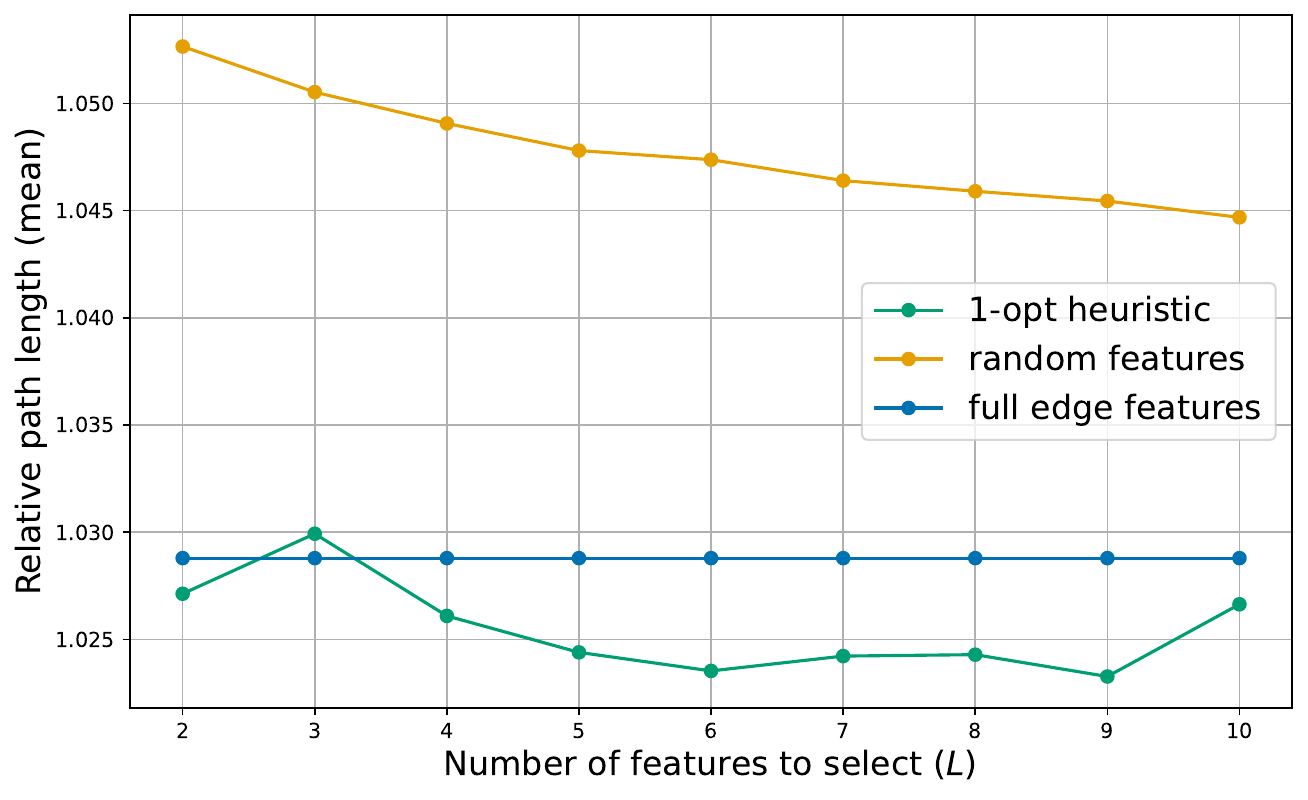}
	\caption{Relative length of the most
explainable path as a function of the cardinality $L$. Results are averaged over 10 runs with $N=200$ data points on the city center graph; the feature list contained 20 grid features as well as all individual edges.}
	\label{fig:exp2}
\end{figure}


\subsubsection{Third Experiment: Full Graph}

We also conducted experiments on the full graph of the city. This time, the graph was divided into a $15 \times 15$ grid, considering only the resulting grid features and excluding any edge features. Since many of these grid cells fall outside the graph, the feature candidate list consisted of 110 grid features. This grid size was chosen to achieve a similar number of features as in the city center graph experiment, ensuring comparability between the two setups.

As shown in Figure~\ref{fig:exp3}, the best explainable routes based on grid features do not outperform the solutions that are best explained by comparing all edge costs. Still, we achieve performance close to the full edge-feature benchmark, while having to compare traffic volumes in only few regions rather than on every individual edge. This substantially improves explainability for the end user. Specifically, we can now provide insights such as, \say{Focus on these regions of the city to compare traffic with other scenarios and identify optimal paths}, rather than analyzing every edge individually.

\begin{figure}[H]
	\centering
	\includegraphics[width=0.5\linewidth]{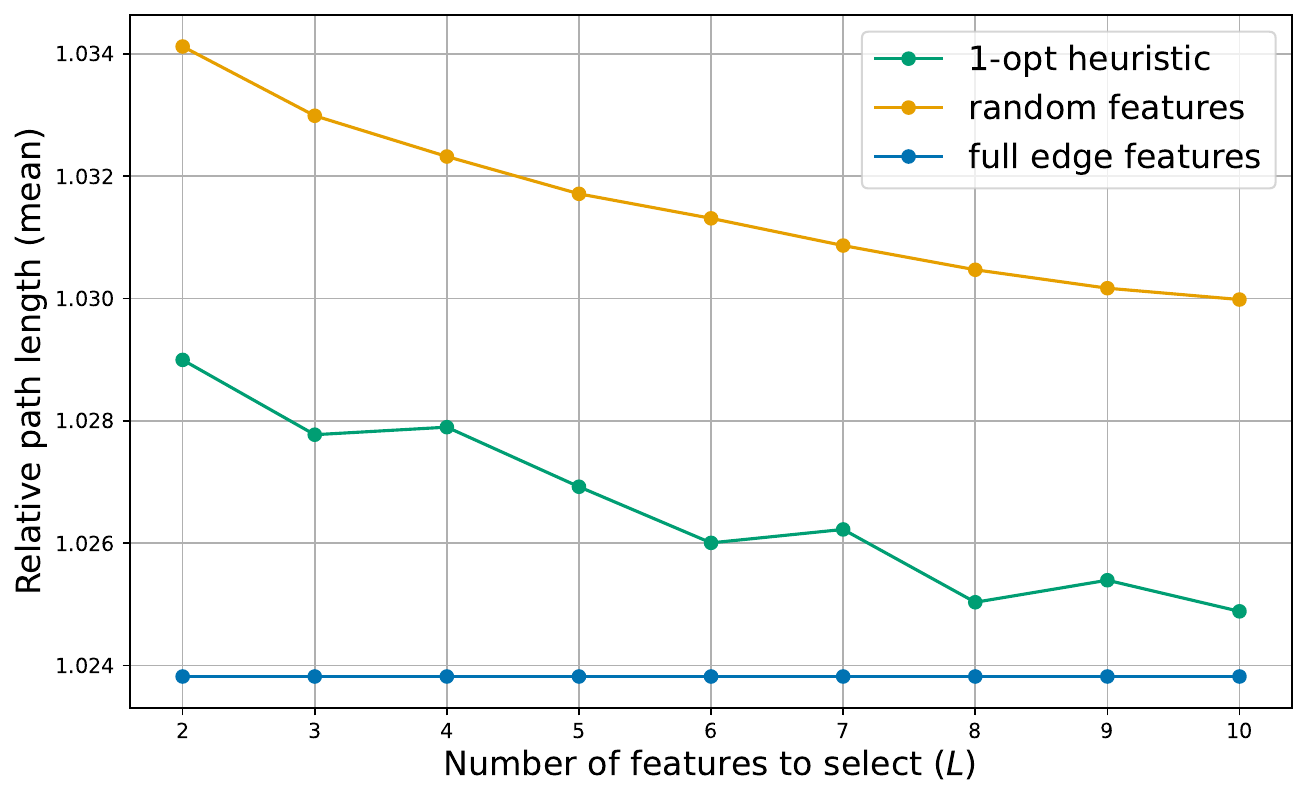}
	\caption{Relative length of the most
explainable path as a function of the cardinality $L$. Results are averaged over 10 runs with $N=200$ data points on the full graph; the feature list contained 110 grid features.}
	\label{fig:exp3}
\end{figure}

\subsubsection{Fourth Experiment: Modified Endpoints}

In the final experiment, we shift both the start and end points of the shortest path to the lower part of the graph. As a result, most edges in the upper region are never used for path selection, and their traffic load is likely uncorrelated with this section of the network (see Figure~\ref{fig:chicago_short}).

\begin{figure}[H]
	\centering
	\includegraphics[width=0.5\linewidth]{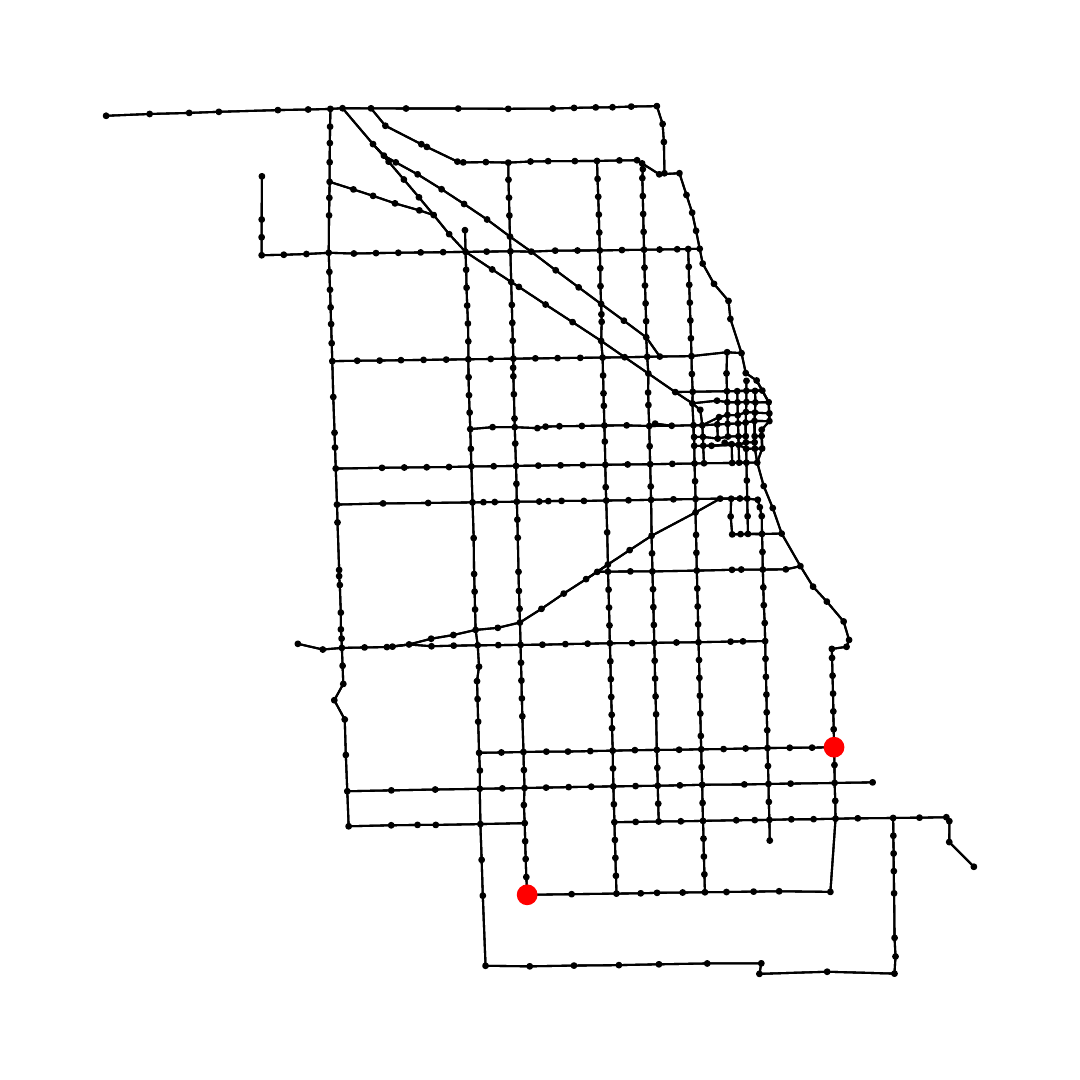}
	\caption{For this experiment, we again used the Chicago instance, but this time, the shortest path is determined between the two red dots in the lower half. As a result, the traffic in most of the network has little to no impact on the pathfinding process.}
	\label{fig:chicago_short}
\end{figure}

We use the same feature set as in the previous experiment, relying exclusively on grid features. However, in this setting, the feature selection approach outperforms the full edge features benchmark (see Figure~\ref{fig:exp4}). This is because it focuses only on features representing traffic in the relevant lower part of the graph, whereas the benchmark still evaluates the overall traffic distribution across the entire network.

This highlights the importance of feature selection in ensuring meaningful comparisons. Without it, irrelevant parts of the network---such as unused edges in this case---can distort the similarity assessment. By selecting only the most relevant features, we improve the explainability of the results and ensure that comparisons focus on the truly influential aspects of the network, leading to more reliable conclusions.

\begin{figure}[H]
	\centering
	\includegraphics[width=0.5\linewidth]{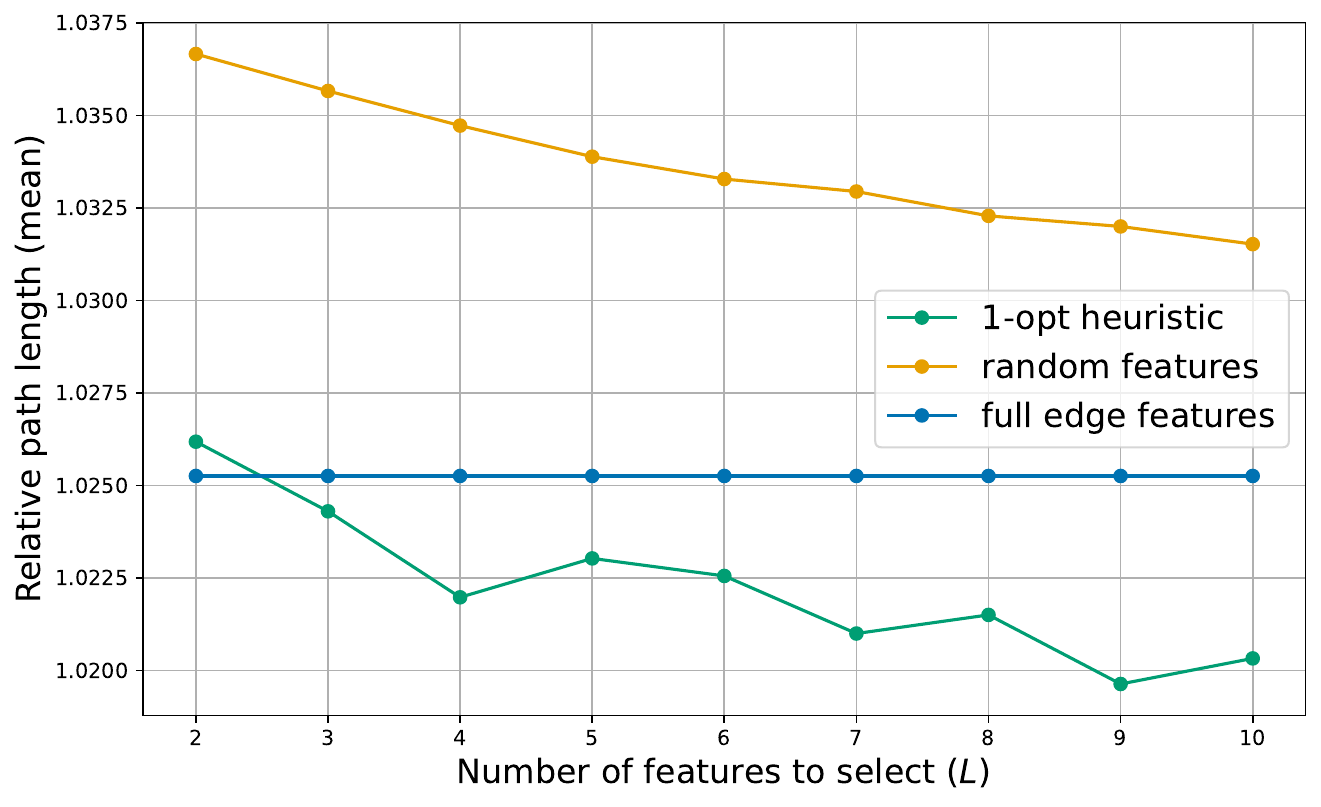}
	\caption{Relative length of the most explainable path as a function of the cardinality $L$. Results are averaged over 10 runs with $N=200$ data points on the full graph but with new start and end points (see Figure~\ref{fig:chicago_short}); the feature list contained 110 grid features.
}
	\label{fig:exp4}
\end{figure}

    \section{Conclusions and Outlook}\label{sec:conclusion}


Explainability of solutions is of key importance when applying optimization methods in practice. The best solution will, in the end, remain worthless if it is not accepted by practitioners. This paper picks up a recently introduced, data-driven concept of explainability in optimization and asks the question how to define features that are used to describe instances. We would like to pick a set of features that is not too large, so that they remain easily understandable. Furthermore, we should design them in a way that they are effective in the explanation process; in our context, this means that instance features are chosen so that similar instances lead to similar solutions.

To break ties between instances that have the same distance, we introduced an optimistic and pessimistic problem variant. We showed that both are NP-hard to solve, and introduced a mixed-integer programming formulation for each.
Due to the problem hardness, we proposed a local search heuristic that iteratively exchanges features to improve the solution quality.  In extensive computational experiments using real-world shortest path data, we compared our feature selection approach with randomly chosen features and with an approach that measures differences on each edge of the graph as benchmarks. Our results show that by using far fewer features than the latter approach, we can obtain a comparable performance in our setting. Even more, if not all data in the instance is relevant, using all edges can even be misleading, and our approach outperforms the benchmark while using fewer features.

Several avenues for further research emerge. A straightforward extension of our proposed concept is to use affine combinations of features, which enables us to weight the importance of features against each other, and even put negative weights on them. While this approach makes the proposed framework more flexible, it also comes at the cost of a more involved mixed-integer programming formulation and a less transparent comparison mechanism within the explanation. Furthermore, this paper only considered the side of instance features, which means that similar considerations for solution features remain open. Additionally, we assume that a set of feature candidates is given, of which we select a small subset. In a further step, we may consider how to find such a set, i.e., generalize from feature selection problems towards feature generation problems, both for solution and for instance features. Finally, the proposed framework of explainability is a first step in this direction, but we believe that many alternative definitions of explainability are conceivable and should be studied in the future.


\section*{Acknowledgments}
We thank the German Research Foundation for their support within Projects B06 and B10 in the Sonderforschungsbereich/Transregio 154 Mathematical Modelling, Simulation and Optimization using the Example of Gas Networks with Project-ID 239904186. This paper has been supported by the BMWK - Project 20E2203B.



\bibliographystyle{informs2014} 
\bibliography{references.bib} 




\end{document}